\documentclass[notitlepage,11pt]{article}
\newcommand{\preprint}{} 

\usepackage{akshay}

\ifdefined \preprint
	\usepackage[left=1.25in,right=1in,top=1in,bottom=1in]{geometry}
	\usepackage{akshayoptional}
\else\ifdefined \informs
	\input{wrapINFORMS.tex}
\else
\fi
\fi

\renewcommandx{\P}[2][2=]{\ifthenelse{\isempty{#2}}{\mathbb{P}\left\{#1\right\}}{\mathbb{P}\left\{#1 \ |\ #2\right\}}} 

\newcommand{\E}[1]{\mathbb{E} #1}
\renewcommand{\deg}[1]{\operatorname{deg}(#1)}

\newcommand{\p}{\beta}
\renewcommand{\r}{r}
\newcommand{\rr}{\gamma}
\newcommand{\pp}{p}

\newcommand{\G}[1][\r]{G^{#1}_{n,p}}

\newcommand{\X}{Y}
\newcommand{\XX}{\mathcal{X}}
\newcommand{\YY}{\mathcal{Y}}
\newcommand{\ZZ}{\mathcal{Z}}
\newcommand{\Exp}[1]{\exp\left\{#1\right\}}
\newcommand{\Stab}{\text{STAB}}
\newcommand{\greedy}[1][]{
\texttt{Gr}
\ifthenelse{\isempty{#1}}{}{\texttt{-min}}
}

\setvalue{\MyTitle}{Large independent sets in \revise{recursive Markov random graphs}}
\setvalue{\MyRuntitle}{Large independent sets in \revise{recursive Markov random graphs}}
\setvalue{\MyAbstract}{
Computing the maximum size of an independent set in a graph is a famously hard combinatorial problem that has been well-studied for various classes of graphs. When it comes to random graphs, only the classical \revise[2]{Erd\H{o}s-R\'enyi-Gilbert} random graph $G_{n,p}$ has been analysed and shown to have largest independent sets of size $\Theta(\log{n})$ w.h.p. This classical model does not capture any dependency structure between edges that can appear in real-world networks. We initiate study in this direction by defining random graphs $\G$ whose existence of edges is determined by a Markov process that is also governed by a decay parameter $\r\in(0,1]$. We prove that w.h.p. $\G$ has independent sets of size $(\frac{1-\r}{2+\epsilon}) \frac{n}{\log{n}}$ for arbitrary $\epsilon > 0$. This is derived using bounds on the terms of a harmonic series, Tur\'an bound on stability number, and a concentration analysis  for a certain sequence of dependent Bernoulli variables that may also be of independent interest. Since $\G$ collapses to $G_{n,p}$ when there is no decay, it follows that having even the slightest bit of dependency (any $\r < 1$) in the random graph construction leads to the presence of large independent sets and thus our random model has a phase transition at its boundary value of $r=1$. \revise{This implies there are large matchings in the line graph of $\G$ which is a Markov random field.} For the maximal independent set output by a greedy algorithm, we deduce that it has a performance ratio of at most $1 + \frac{\log{n}}{(1-\r)}$ w.h.p. when the lowest degree vertex is picked at each iteration, and also show that under any other permutation of vertices the algorithm outputs a set of size $\Omega(n^{1/1+\tau})$, where $\tau=1/(1-\r)$, and hence has a performance ratio of $O(n^{\frac{1}{2-\r}})$.
}
\setvalue{\MyKeywords}{Independent sets, Greedy algorithm, Concentration inequalities, Tur\'an's theorem, Dependent Bernoulli sequence}
\setvalue{\MySubjclass}{90C27, 60J10, 05C80, 05C69}
\setvalue{\MyDate}{Submitted August 2022; Revised December 2023, April 2024}

\setvalue{\NumAuthors}{2}

\renewcommand{\authordef}[1]{%
\ifcase#1\relax \setvalue{\authorname}{Akshay Gupte} \setvalue{\authoraffil}{School of Mathematics and Maxwell Institute for Mathematical Sciences, The University of Edinburgh, UK} \setvalue{\authoremail}{akshay.gupte@ed.ac.uk}

\or \setvalue{\authorname}{Yiran Zhu} \setvalue{\authoremail}{yiran.zhu@ed.ac.uk}

\else Jane Doe.
\fi
}

\begin{document}


\ifdefined \preprint
	\author{\authordef{0}\MyAuthor \and \authordef{1}\MyAuthor[1]}
	\renewcommand{\revise}[2][]{#2}
	\renewcommand{\changes}[1]{#1}
	\title{\MyTitle}
\date{\MyDate}

\maketitle

\begin{abstract}
{\MyAbstract}

\mykeywords{\MyKeywords}

\mysubjclass{\MySubjclass}
\end{abstract}
\else\ifdefined \informs
	\MANUSCRIPTNO{MOR-2022-215-R2}
	\ARTICLEAUTHORS{
	\AUTHOR{Akshay Gupte, Yiran Zhu}
	\RUNAUTHOR{Gupte and Zhu}
	\AFF{School of Mathematics and Maxwell Institute for Mathematical Sciences, University of Edinburgh, UK, \EMAIL{\{akshay.gupte, yiran.zhu\}@ed.ac.uk}}
	}
	\AREAOFREVIEW{Discrete Optimization}
	\FUNDING{The initial phase of this research was supported by NSF grant DMS-1913294.}
	\maketitle
\else
\fi
\fi

\section{Introduction}	\label{sec:intro}

An independent set in a graph $G = (V,E)$ is a subset of $V$ such that no two vertices in this subset have an edge between them.  The maximum cardinality of an independent set in $G$ is called the stability number $\alpha(G)$, and this is a difficult combinatorial problem that is strongly NP-hard to compute exactly and also to approximate within factor arbitrarily close to $\abs{V}$. 
There is vast amount of literature on approximating this number for general and special graphs, in theory and also computationally through optimization algorithms \cite[cf.][]{bomze1999maximum,gaar2022sdp,galluccio2009gear,Rebennack2009branch}. Graphs generated through some randomisation technique are natural candidates for analysis of many structural graph properties. The most basic class of random graphs is the classical \revise[2]{Erd\H{o}s-R\'enyi-Gilbert}
 random graph, also called the \revise[2]{Erd\H{o}s}-R\'enyi-Gilbert random graph. This is denoted by $G_{n,p}$ and is a $n$-vertex graph where each edge has a given fixed probability $p$ of being present in the graph.\footnote{To be precise, the random graph model proposed by \revise[2]{Erd\H{o}s} and R\'enyi fixes the number of edges instead of fixing the edge probability, but it is well-known that the two definitions are equivalent.} There is rich history on bounding $\alpha(G_{n,p})$ asymptotically  and it is known that the largest independent sets in $G_{n,p}$ are of $\revise{\Theta(\log{n})}$ with high probability (w.h.p.) \cite{bollobas1976cliques,frieze1990independence,matula1976largest}. Bounds have also been computed by solving convex optimization problems such as those from the Lov\'asz $\vartheta$-function and the Lov\'asz-Schrijver lift-and-project operator, 
and the latter produce a tight relaxation w.h.p. after being applied for $\Theta(\log{n})$ rounds on $G_{n,1/2}$ \cite{feige2003probable}. The analysis of $\alpha(G_{n,p})$ has spawned interest in the chromatic number $\chi(G_{n,p})$, which is related as per the relation $\chi(G)\alpha(G)\ge n$ for any $n$-vertex graph $G$, 
and its concentration \cite[cf.][]{coja2008chromatic,heckel2021non,mcdiarmid1990chromatic}. 
Many other properties of $G_{n,p}$ have also been studied in great detail; see recent texts such as the monograph \cite{Frieze2016} and \cite[chap.~10]{Bremaud2017Discrete}. 

We continue this line of work on asymptotic analysis of $\alpha(G)$ but initiate it on a new class of random graphs for which $G_{n,p}$ is a boundary condition. Random graph models have been used for analyzing topological structures of networks, such as social network, citation network, \changes[2]{internet}, etc. \cite{newman2002random,barabasi2016network,fosdick2018configuring,robins2007introduction}, and independent sets present some important properties of these networks \revise{since they represent the group of people who are not known to each other. There are also some specific examples such as the Nash equilibrium in public goods networks being formulated using the maximal independent set \cite{bramoulle2007public}. }  Typical network models have non-uniform edge probability and  nonzero correlations between edges, \revise[2]{and any such variation makes the random graph inhomogeneous. There are various ways of constructing these \cite[cf.][chap.~1]{bollobas2008handbook}, usually through a recursive generation of the edges with some underlying pattern. A common family is the preferential attachment model due to \citet{barabasi1999emergence} where each new vertex  in the generation process is more likely to be connected to an existing vertex which already has many edges. They are sometimes also referred to as scale-free random graphs. Many of their combinatorial properties, such as degree distributions, connectivity, diameter, etc. have theoretical guarantees \cites{mihail2006certain}[][chap.~1--4]{bollobas2008handbook}[][chap.~16 and 19]{Frieze2016}, which can lead to bounds on the stability number by exploiting connections (in any graph) between certain graph parameters and sizes of independent sets. Another class of inhomogeneous random graphs which has been gaining traction with its many applications in social networks is the exponential random graph model \cite{robins2007introduction}. These have a complex generation process that allow control over some graph parameters, such as number of triangles, stars, etc. Here, there have been empirical studies in practical networks \cite{robins2009closure}, but we are not aware of theoretical bounds on sizes of independent sets. However, there are some random graph families, having a construction that imposes implicit dependency between edges, whose stability or clique number has been explicitly characterized}, such as random regular graphs \cite{Ding2016},   \revise[2]{those generated through hyperbolic geometry \cite{blasius2018cliques}, and those where edge probability is a scaled product of random (iid) vertex weights with some distribution \cite{bogerd2020cliques,janson2010large}.} This paper creates a new \changes[2]{inhomogeneous} random graph model \changes[2]{possessing an explicit dependence between edges and} that can be further scrutinised for many of its properties as has been done for the \changes[2]{other models from literature}.

The edges in our random graph are generated dynamically using a Markov process. Given $n$, $p$ and a decay parameter $\r\in(0,1]$, starting from the singleton graph $(\{v_{1}\},\emptyset)$,  a graph $\G$ having $n$ vertices is generated in $n-1$ iterations where at each iteration $t\ge 2$, the vertex $v_{t}$ is added to the graph and edges $(v_{i},v_{t})$ for $1\le i \le t-1$ are added as per a Bernoulli r.v. $X_{i}^{t}$. The success probability of $X_{i}^{t}$ is equal to $p$ for $i=1$ and for $i\ge 2$, it is independent of the values of $\{X^{t}_{1},\dots,X^{t}_{i-2} \}$ and is equal to the success probability of $X_{i-1}^{t}$ when $X^{t}_{i-1}=0$, and is reduced by a factor $\r$ when $X^{t}_{i-1}=1$.


\begin{definition}\label{def:markov-graph}
For any $p, \r \in (0,1]$, a \emph{\revise{recursive Markov random graph}} $\G$ is a random graph on $n$ vertices $\{v_{1},\dots,v_{n}\}$ wherein the edge probabilities are defined as follows : 
\begin{enumerate}
\item for $2\le i \le n$, edge $(v_{1},v_{i})$ exists with probability equal to $p$,
\item for $2 \le j < i \le n$, probability of edge $(v_{j},v_{i})$ depends on whether edge $(v_{j-1},v_{i})$ is present or not, and is independent of all other edges, and we have the following dependency structure on their conditional probabilities, 
\begin{equation*}	\label{recurrence1}
\begin{split}
\P{\text{edge $(v_{\revise{j}},v_{i})$ exists}}[\text{edge $(v_{\revise{j-1}},v_{i})$ does not exist}] &= \P{\text{edge $(v_{\revise{j-1}},v_{i})$ exists}} \\
\P{\text{edge $(v_{\revise{j}},v_{i})$ exists}}[\text{edge $(v_{\revise{j-1}},v_{i})$ exists}] &= \r \, \P{\text{edge $(v_{\revise{j-1}},v_{i})$ exists}}.
\end{split}
\end{equation*}
\end{enumerate}
\end{definition}
We think of the parameter $p$ as the \emph{initial probability} since it governs the edge probabilities between the first vertex and other vertices, and the parameter $\r$ as the \emph{decay parameter} since it decreases the edge probability by this factor whenever a previous edge (think of the vertices sorted from $1,\dots,n$ and edges between them) is present. Taking $\r=1$ makes all edges have equal probability $p$ and so $\G[1]$ is isomorphic to $G_{n,p}$, thereby making our \revise{recursive} Markov random graph a generalization of the \revise{classical} \revise[2]{Erd\H{o}s}-R\'enyi-Gilbert model.

\revise{\citet[][pp. 835]{Frank1986MarkovGraph} first defined a {Markov random graph} as a random graph whose edges are generated by a stochastic process which is such that any two edges $(v_{i},v_{j})$ and $(v_{k},v_{l})$ that do not share a common vertex (i.e., $i,j,k,l$ are distinct) are conditionally independent given all the other edge variables. The construction of our random graph $\G$ makes it clear that non-adjacent edges are conditionally independent, and so it is a Markov random graph. But it is a special subclass because of how the edges are generated in a recursive manner and has much greater independence properties. In particular, the Markov process used for its generation means that every edge is also independent of co-incident edges when they are not immediate neighbours, i.e., $(v_{j},v_{i})$ is independent of $(v_{k},v_{i})$ for $k < j-1$, and of $(v_{i},v_{k})$ for $k > i$. We note that in the line graph of the \citeauthor{Frank1986MarkovGraph} model, \revise[2]{the Bernoulli variable for any vertex is conditionally independent of all the vertices outside its neighbourhood, implying that the random variables associated with these vertices can be interpreted to form a Markov random field} \cite[cf.][chap.~9]{Bremaud2017Discrete}. In the line graph of our model $\G$, every vertex is independent of all vertices except exactly one adjacent vertex.}

\ifdefined \preprint
	\paragraph{Outline.}
\else
	\subsubsection*{Outline}
\fi

\changes{
Since our random graph model is a generalisation of the classical \revise[2]{Erd\H{o}s-R\'enyi-Gilbert}
 model, the mathematical question naturally arises as to which classical properties extend to our random model. We make the first attempt at understanding this by analysing independent sets, which were  one of the first structures studied in the classical model. We provide two lower bounds and an upper bound on the stability number $\alpha(\G)$. These also imply bounds on the sizes of matchings in the line graph of $\G$.

Our main results are stated in \Cref{sec:results}, along with some of their consequences and discussion of proof techniques, and we also point to parts of the paper where these are proved. Some open questions are also mentioned. \Cref{sec:chain} notes some basic properties of the Bernoulli random variables associated with edges in the graph, including making the crucial observation that there is essentially a single Markov chain that represents all the edges. General technical results that are building blocks of the analysis in this paper are presented in \Cref{sec:technical}, with their proofs at the end in \Cref{sec:missing,sec:bernoulli}. Some of these may also be more widely applicable and of independent interest. The first lower bound is proved in \Cref{sec:lower} using average vertex degrees in the graph, and the second one is in \Cref{sec:greedy} after analysing the performance of the greedy algorithm for computing maximal independent sets. An upper bound on the stability number is derived in \Cref{sec:upper}. In between, \Cref{sec:edge-prob} bounds the various probabilities that arise in our analysis.
}




\ifdefined \preprint
	\paragraph{Notation.}
\else
	\subsubsection*{Notation}
\fi


The random graphs $\G$ form an obvious finite probability space. Random variables in this paper are generally some function of the random graph $\G$. We suppress $\omega$ from a r.v. $\xi(\omega)$ and write it simply as $\xi$. Commonly appearing random variables are the stability number $\alpha(\G)$, vertex degree which is $\deg{v_{i}}$, average degree $d(\G)$, edge $(v_{j},v_{i})$ Bernoulli r.v. $X^{i}_{j}$. We use $\E{}$ to take expected value and $\P{}$ for probability measure, where the corresponding probability space is obvious from context. When a sequence of random variables $\{\xi_{n}\}$ on the same probability space converges to another random variable $\xi$, we denote it as $\xi_{n} \overset{p}{\longrightarrow} \xi$, and when $\xi$ is a constant we sometimes say that $\xi_{n}$ concentrates to $\xi$. \revise{This convergence in probability means that for every $\epsilon > 0$, the real sequence $\P{\abs{\xi_{n} - \xi} \le \epsilon}$ converges to 1 as $n\to\infty$.}
A property $\mathcal{P}_n$ is said to be with high probability (w.h.p.) if $\P{\mathcal{P}_n \text{ is true}}\to 1$ as $n\to\infty$. Logarithm to the base $e$ (natural log) is denoted by $\log{}$. Since the constant $1-\r$ will appear many times in this paper, we denote
\begin{equation} \label{def:gamma}
\rr := 1 - \r.
\end{equation}

\section{Summary of Results} \label{sec:results} 

We assume throughout that $\r < 1$, and so our results do not apply to the \revise[2]{Erd\H{o}s-R\'enyi-Gilbert}
 random graph. Our first main result is an asymptotic lower bound on $\alpha(\G)$ in terms \revise[2]{of $n$}.

\begin{theorem} \label{lower}
For every $\epsilon > 0$, we have w.h.p. that
\[
    \alpha(\G) \ge \frac{\rr }{2+\epsilon} \, \frac{n}{\log{n}}.
\]
\end{theorem}

A key ingredient of this proof in \Cref{sec:lower} is establishing a concentration result about the average vertex degree $d(\G)$. 

\begin{theorem}	\label{avgdeg}
$d(\G)/\log{n}$ concentrates to $\frac{2}{\rr}$.
\end{theorem}

The proof for this requires analysing the sequence of dependent Bernoulli r.v.'s corresponding to each vertex $v_{i}$. This analysis is carried out in \Cref{sec:bernoulli}, but we use \Cref{bernoulli-main} and \Cref{bernoulli-main2} corresponding to it when proving our above concentration result.

Another lower bound can be obtained from the maximal independent set produced by a greedy algorithm. Recall that a greedy algorithm sorts the vertices in some order and then iteratively adds them to the output set if the addition retains the independent set property for the subset. In \Cref{sec:greedy}, \revise{we analyze the greedy algorithm in two ways. First we show that running it on the permutation of vertices from $v_{1}$ to $v_{n}$ yields a lower bound that is weaker than the $\Omega(n/\log(n))$ bound in \Cref{lower}.}

\begin{proposition} \label{greedy}
The greedy algorithm \revise{when run on the sequence $\{v_{1},v_{2},\dots,v_{n}\}$} outputs w.h.p. an independent set of size $\Omega\left(n^{\frac{\rr}{\rr+1}}\right)$.
\end{proposition}

\changes{The fixed sequence $v_{1}$ to $v_{n}$ may not be the best strategy for the greedy algorithm. In fact, to maximise the size of the maximal independent set output by the greedy algorithm, one would think of choosing the smallest degree vertex at each iteration. For this strategy, we deduce in \Cref{greedyratio} that the output of greedy is $O(\log{n})$-factor away from the stability number.}

 
For upper bound on the size of independent sets in $\G$, we prove in \Cref{sec:upper} a nontrivial  constant $c<1$ that bounds $\alpha(\G)\le c\,n$.
 
\begin{theorem}	\label{upper}
We have w.h.p. that $\alpha(\G) \le \left(e^{-\r} + \frac{\r}{10}\right)\, n$.
\end{theorem}

A crucial first step in our entire analysis is to establish the edge probabilities, i.e., success probability for Bernoulli r.v. corresponding to each edge. This is not a straightforward task, unlike the \revise[2]{Erd\H{o}s-R\'enyi-Gilbert}
 graph $G_{n,p}$ for which this probability is readily available as the input parameter $p$. Although we don't derive an exact expression for the edge probabilities, we derive tight lower and upper bounds on it in \Cref{sec:edge-prob}. This is done by analyzing in \Cref{sec:recur} the rate at which terms in the recurrence formula $f_{a}: x\in\real \ \mapsto \ x(1-ax)$ grow. This recurrence appears due to the Markov process that generates our random graph.

\subsection{Some Consequences}	\label{sec:conseq}

For any $n$-vertex graph $G$ with $m$ edges, the average degree $d(G)$ is equal to $2m/n$ and we have $\delta(G) \le d(G) \le \Delta(G)$. Hence, \Cref{avgdeg} implies that $\G$ has w.h.p. approximately $\Theta(n\log{n})$ many edges, and there exist vertices with degrees $O(\log{n})$ and those with degrees $\Omega(\log{n})$, where the constants hiding in this notation are linear in $\r$.

\begin{corollary}	\label{conseqdeg}
For every $\epsilon > 0$, we have w.h.p. that $\delta(\G) \le \left(\frac{2}{\rr} +\epsilon \right)\log{n}$ and $\Delta(\G) \ge \left(\frac{2}{\rr} -\epsilon \right)\log{n}$.
\end{corollary}

The chromatic number $\chi(G)$ of any $n$-vertex graph $G$ is lower-bounded by $\alpha(G)$ through the relation $\chi(G) \ge n/\alpha(G)$, which means that an upper bound on $\alpha(G)$ provides a lower bound on $\chi(G)$. Since our upper bound on $\alpha(\G)$ is $O(n)$, this basic relation does not tell us anything useful. Instead, we use the Tur\'an bound on chromatic number to deduce a lower bound on $\chi(\G)$ in terms of the prime counting function $\pi(n)$. We also note a lower bound on the edge-chromatic number $\chi^{\prime}$.

\begin{corollary}
For every $\epsilon > 0$, we have w.h.p. that
\[
\chi(\G) \ge \frac{\pi(n)}{\pi(n) + \epsilon - \frac{2}{\rr}}, \qquad \text{ and } \quad \chi^{\prime}(\G) \ge \left(\frac{2}{\rr} -\epsilon \right)\log{n} .
\]
\end{corollary}
\begin{proof}
The lower bound on $\chi^{\prime}$ is immediate from \Cref{conseqdeg} and the fact that $\chi^{\prime}(G)\ge\Delta(G)$ for any graph $G$. The bound on $\chi$ is using the bound from Turan-type arguments \cite[Theorem 8]{coffman2003bounds} which says that $\chi(G) \ge n^{2}/(n^{2} - 2m)$ for any graph $G$ with $n$ vertices and $m$ edges. Since $d(G) = 2m/n$ makes this bound equal to $n/(n-d(G))$, we obtain that 
\[
\chi(\G) \ge \frac{n}{n-d(\G)} \eq \frac{1}{1 - \frac{d(\G)/\log{n}}{n/\log{n}}} \ge \frac{1}{1 - \frac{2/\rr - \epsilon}{\pi(n)}} \quad \text{w.h.p.} \ \forall\,\epsilon > 0,
\]
where in the last inequality we have used the concentration from \Cref{avgdeg} and the fact that $\pi(n) \ge n/\log{n}$ for $n \ge 17$.
\end{proof}

Although this lower bound on $\chi(\G)$ is not a constant, it is clear that it is a weak bound and so there remains scope for improvement on it.

Another question related to chromatic number that has raised interest for the \revise[2]{Erd\H{o}s-R\'enyi-Gilbert}
 random graph $G_{n,p}$ is that of its Hadwiger number $h(G_{n,p})$, which for a graph $G$ is the largest value of $t$ such that $G$ contains a $K_{t}$-minor. This is related to the chromatic number $\chi(G)$ because of the famous Hadwiger's conjecture that $\chi(G) \le h(G)$ for any graph $G$. \citet{erde2021large} showed recently that $h(G_{n,p}) = \Omega(\sqrt{k})$ when $p=(1+\epsilon)/k$, where $k$ is a lower bound on $\delta(G_{n,p})$. Although an analysis of the Hadwiger number of $\G$ requires separate attention that is beyond the scope of this paper, we note some observations here. Recall that the line graph $L(G)$ of a simple graph $G$ is obtained by vertex-edge duality.

\begin{corollary}
For every $\epsilon > 0$, w.h.p. there exists a $K_{t}$-minor in $\G$ for $t=\Omega_{\epsilon}(\frac{\log{n}}{\sqrt{\log{\log{n}}}})$ and a $K_{s}$-minor in $L(\G)$ for $s = \Omega_{\epsilon}(\log{n})$, 
where $\Omega_{\epsilon}$ hides a constant in terms of $\epsilon$.
\end{corollary}
\begin{proof}
A graph $G$ with average degree $d(G) \ge k$ has $h(G) = \Omega(k/\sqrt{\log{k}})$ \cite[cf.][Corollary~2.2]{erde2021large}. This bound is a convex function of $k$ that is increasing for $k\ge 2$. Since the average degree concentration in \Cref{avgdeg} tells us that $d(\G) \ge (\frac{2}{\rr}-\epsilon)\log{n}$ w.h.p., it follows that $\G$ has a complete minor of size at least $\Omega_{\epsilon}(\frac{\log{n}}{\sqrt{\log{\log{n}}}})$. For line graphs of simple graphs, it is well-known that Hadwiger's conjecture is true as a direct consequence of Vizing's theorem which says that $\chi^{\prime}(G) \in \{\Delta(G), \Delta(G)+1\}$ and because $\chi(L(G))=\chi^{\prime}(G)$. Hence, we have $h(L(\G)) \ge \chi(L(\G)) = \chi^{\prime}(\G) \ge \Delta(\G) \ge (\frac{2}{\rr}-\epsilon)\log{n}$, where the last inequality is from \Cref{conseqdeg}.
\end{proof}

\changes{A final straightforward implication is that $L(\G)$ has large matchings of the order $\Omega(n/\log{n})$ w.h.p., because an independent set in a graph corresponds to a matching in the line graph.}

\subsection{Discussion} \label{sec:discuss}

The stability number of \revise[2]{Erd\H{o}s}-R\'enyi-Gilbert random model $G_{n,p}$ has been well-studied. If $p\in (0,1)$ is fixed, then \citet{matula1976largest} has shown that w.h.p. the stability number $\alpha(G_{n,p}) \approx 2\log_{q}(n)$, where the logarithm base is $q = 1/(1-p)$. The graph $G_{n,p}$ with fixed $p$ is usually considered as a dense random graph. For sparse random graphs with $np= d$ fixed, \citet{frieze1990independence} showed that
\begin{equation*} 
	\left|\alpha(G_{n,p}) - \frac{2n}{d}\left(\log d - \log\log d- \log2+1 \right) \right| \le \frac{\epsilon n}{d}
\end{equation*}
holds w.h.p. for all $d\ge d_\epsilon$ where $d_\epsilon$ is a constant depending on the chosen $\epsilon >0$. In contrast, our lower bound of $\Omega(n/\log{n})$ from \Cref{lower} indicates that there are much larger independent sets in our random graph $\G$ when $p$ and $\r$ are fixed. Since all of our analysis heavily depends on $\r < 1$, our results don't generalise those for $G_{n,p}$, thus implying that a phase transition occurs in our random graph model at the boundary value $\r=1$.

The $\log{n}$ term in our main results comes from bounding the partial sum of a harmonic series that arises due to the Markovian dependence between the edge probabilities. The lower bound in \Cref{lower} makes use of Tur\'an's lower bound on $\alpha(G)$ for any graph $G$. 
 The crucial part of our analysis is the average degree concentration in \Cref{avgdeg}. Our proof for this theorem uses Chebyshev's inequality. Due to the absence of independence structure between the r.v.s, we cannot apply Chernoff- or Hoeffding-type inequalities, and use of martingale tail inequalities also does not help. For fixed $p$, this concentration result shows $\G$ to be more sparse than $G_{n,p}$ in terms of the number of edges. Intuitively, a denser graph has smaller stability number and our result indeed complies with this intuition,
\[
\alpha(\G) \eq \Omega\left(\frac{n}{\log n} \right) \eq \Omega(\log_{\frac{1}{1-p}}(n)) \eq \Omega(\alpha(G_{n,p})) .
\]
\Cref{avgdeg} relies on the asymptotic result of the sequence $\frac{S_n}{\log n}$ in \Cref{bernoulli-main} \revise{where $S_n$ is the partial sum of a non-i.i.d. Bernoulli sequence in which the probabilities satisfy a certain recurrence relation}. This result does not trivially follow from the literature. We know of two studies \cite{JamesLimit,LanAsymptotics} that show $S_n/n$ converges to constants under the assumption that $\P{\X_{n+1} = 1}[\X_{1},\dots,\X_{\revise{n}}] = \theta_n + g_n(S_n)$ where $g_n$ is a linear function and $\theta_n$ is a constant. Our work makes a  different assumption where the conditional probability does not depend on the partial sum $S_n$. Furthermore, the variables we consider form a nonhomogenous Markov chain and to the best of our knowledge such a sequence has not been analysed before, although we mention that partial sums of a Bernoulli sequence with a homogenous Markov chain has been studied before \cite{edwards1960meaning,Wang1981}. 
Another powerful method to deal with a sequence of variables is martingale analysis. It is easy to see $\{S_n\}_{n\ge1}$ is a super-martingale because $\E{\left[S_{n+1}\,|\,Y_1,\ldots,Y_n\right]} = \E{\left[Y_{n+1}\,|\,Y_n\right]} +S_n > 0$. Since $\{S_n\}_{n\ge1}$ is non-decreasing, Doob's inequality collapses to Chebyshev's inequality $\P{S_n>C}\le \E{S_n}/C$. So it appears  that a martingale technique would not improve our analysis.

\changes{We finish our discussion by stating some open questions that follow immediately from our main results.} The first of these is whether our lower bound of $n/\log{n}$ is tight.

\begin{question}
Does there exist a function $f(n) = \Omega(n/\log{n})$ such that $f(n) = o(n)$ and $\alpha(\G) = \Theta(f(n))$ w.h.p.?
\end{question}

\changes{Our lower bound from \Cref{lower} can also be stated in terms of the prime-counting function, as seen later in \Cref{primes}. This suggests a possible algorithm for computing independent sets in $\G$ by testing whether each number $k \le n$ is prime or composite (and it is well-known that this check can be done in time polynomial in $\log{k}$) and using the subset of vertices corresponding to primes less than $n$. We do not know what the probability is for such a subset to be independent, nor is it clear that even asymptotically this algorithm yields independent sets (note that having a lower bound from the prime-counting function is a weaker result than the algorithm yielding independent sets). Answering these questions may involve exploring the literature on distributions of primes.

\begin{question}
Does the above algorithm yield independent sets w.h.p.?
\end{question}

A third question arises from the greedy algorithm for which we do performance analysis in \Cref{greedy} and \Cref{greedyratio}. For \revise[2]{Erd\H{o}s-R\'enyi-Gilbert}
 random graphs, this heuristic generally works well empirically \cite{goldberg2005experimental} and also has been shown to be theoretically almost optimal amongst local search algorithms when the average degree is a constant \cite{coja2015independent}. Our concentration result for average degree in \Cref{avgdeg} indicates that an analogous question in our context would be the following.

\begin{question}
When the average degree of $\G$ is fixed to $c\cdot\log{n}$ for some constant $c$, is the greedy algorithm almost optimal amongst all local search algorithms? Does it have a performance ratio $\Omega(\log{n})$? 
\end{question}
}


\section{\revise{Preliminaries on $\G$}}	\label{sec:chain} 

In any random graph model, each edge can be associated with a Bernoulli random variable (r.v.) which is equal to 1 if and only if that edge is present in the random graph. The success probability for this r.v. depends on the construction of the random graph model; for the \revise[2]{Erd\H{o}s-R\'enyi-Gilbert}
 random graph $G_{n,p}$ by \revise[2]{Erd\H{o}s}-R\'enyi-Gilbert, this probability is equal to the parameter $p$, which implies that any two Bernoulli r.v.'s are iid. This significantly helps the analysis of these graphs for their stability number, concentration of chromatic number and  many other graph properties. The same is not true for arbitrary edges in our random graph $\G$ since we have a dependency structure between the edges. Denote the Bernoulli r.v. $X^{i}_{j}$ by 
\[
X^{i}_{j} \eq 1 \iff \text{edge $(v_{j},v_{i})$ is present in graph $\G$}, \qquad 1 \le j < i \le n.
\]
and its success probability by 
\[
\pp^{i}_{j} \define \P{X^{i}_{j}=1}, \qquad 1 \le j < i \le n.
\]

\revise{It is clear that $\pp^{i}_{1} = p$ for every $i$. For the other edge probabilities, we don't have a closed-form expression in terms of $p$. Instead, the generation of our random graph implies a recursive equation which also leads to lower and upper bounds on the edge probabilities. Recall $\ell_{j}$ from \cref{piecewise}.}

\begin{lemma}	\label{recur}
For every $2 \le j \le i-1$, we have
\[
\pp^{i}_{j} \eq \pp^{i}_{j-1}\left[1 \,-\, \rr \pp^{i}_{j-1}\right].
\]
\end{lemma}
\begin{proof}
The recursion follows from the law of total probability,
\begin{align*}
		\P{X^{i}_{j}= 1} &\eq \P{X^{i}_{j}= 1}[X^{i}_{j-1}=1]\P{X^{i}_{j-1}=1} \: + \: \P{X^{i}_{j}= 1}[X^{i}_{j-1}=0]\P{X^{i}_{j-1}=0} ,\\
		&\eq \r\,\P{X^{i}_{j-1}=1}\P{X^{i}_{j-1}=1} \:+\: \P{X^{i}_{j-1}=1}\P{X^{i}_{j-1}=0} ,\\
		&\eq \r(\pp^{i}_{j-1})^{2} \:+\: \pp^{i}_{j-1}(1-\pp^{i}_{j-1}) ,\\
		&\eq \pp^{i}_{j-1}\left[1 \,-\, \rr \pp^{i}_{j-1}\right] ,
\end{align*}
where the second equality is from the conditional probabilities in \Cref{def:markov-graph}. 
\end{proof}

\revise{Each vertex $v_{i}$, for $i\ge 2$, is associated with a Bernoulli sequence $\{X^{i}_{1},\dots,X^{i}_{i-1} \}$ corresponding to the edges from $v_{i}$ to the previous vertices $\{v_{1},\dots,v_{i-1}\}$}. It is obvious from the construction of our random graph that $\{X^{i}_{j}: j\le i-1 \}$ forms a non-homogenous Markov chain whose transition matrix from $X^{i}_{j-1}$ to $X^{i}_{j}$ can be stated as
\begin{equation}	\label{def:Pik}
\mathbf{P}^{i}_{j-1} \eq   \begin{bmatrix}
		1-\pp^{i}_{j-1} \;&\quad \pp^{i}_{j-1} \medskip \\
		\revise{1 - \r}\pp^{i}_{j-1} \;&\quad \r\pp^{i}_{j-1}
	\end{bmatrix}, \qquad 2 \le j \le i-1.
\end{equation}
\changes{We note that there is a \emph{single} Markov chain with a similar transition matrix as \cref{def:Pik} and which represents \emph{all} the edges in the graph, not just the ones corresponding to a single vertex. To state this result, we use the terminology that two Markov chains $\{\mathcal{X}_{n} : n\ge 1 \}$ and $\{\mathcal{Y}_{n} : n\ge 1  \}$ are \emph{equivalent} to mean that  (a) $\mathcal{X}_{1}$ and $\mathcal{Y}_{1}$ have the same distribution, and (b) for each $i\ge 1$, the transition matrix from $\mathcal{X}_{i}$ to $\mathcal{X}_{i+1}$ is equal to that from $\mathcal{Y}_{i}$ to $\mathcal{Y}_{i+1}$.

\begin{lemma}	\label{markovchain}
The Markov chains corresponding to any two vertices $v_{i}$ and $v_{j}$ in $\G$, with $2 \le i < j$, are such that $\{X^{i}_{1}, X^{i}_{2}, \ldots, X^{i}_{i-1}\}$ is equivalent to $\{X^{j}_{1}, X^{j}_{2}, \ldots, X^{j}_{i-1}\}$.

In particular, if $\{\mathcal{X}_{1}, \mathcal{X}_{2},\dots,\mathcal{X}_{n}\}$ is a Markov chain of Bernoulli variables where the success probabilities are $\P{\mathcal{X}_{i}=1} = p_{i}$ with $p_{1} = p$, and the transition matrix from $\mathcal{X}_{j-1}$ to $\mathcal{X}_{j}$ is given by \[\mathbf{P}_{j-1} = \begin{bmatrix} 
	1 - p_{j-1} \;&\quad p_{j-1} \medskip \\
	1 - \r p_{j-1} \;&\quad \r p_{j-1}
\end{bmatrix} ,\]
then for every $i\ge 2$, the subchain $\{\mathcal{X}_1,\ldots, \mathcal{X}_{i-1}\}$ is equivalent to the chain $\{X^{i}_{1}, X^{i}_{2}, \ldots, X^{i}_{i-1}\}$ for $v_{i}$.
\end{lemma}
\begin{proof}
The second assertion follows from the arguments for the first assertion after recognising that the transition matrix of the new chain has the same structure as \cref{def:Pik}. For the first claim, it suffices to show that the edges from $v_{i}$ to the subsequent vertices $\{v_{i+1},\dots,v_{n}\}$ are equal in distribution and also independent of each other. 


Let us argue that $X^{j}_{i}$ and $X^{k}_{i}$ are iid for $ 1\le i < j < k$. When generating $\G$, each iteration $i \ge 1$ uses a Markov chain to produce the edges $\{(v_{t}, v_{i}) \colon t < i \}$ to vertex $v_{i}$. The dependency between the edges is limited to each iteration, meaning that we have intra-iteration dependency but inter-iteration independency.  Since $X^{j}_{i}$ and $X^{k}_{i}$ belong to distinct iterations $j$ and $k$, it follows that they are independent of each other. It remains to show they are identical in distribution, i.e., they have the same success rate. This can be argued by induction on $i$. The base case $i=1$ is obvious because by construction, both $X^{j}_{1}$ and $X^{k}_{1}$ have a success probability of $\pp$. For $i\ge 2$, the claim follows from the induction hypothesis $\pp^{j}_{i-1} = \pp^{k}_{i-1}$ and the recursive equation from \Cref{recur}.
\end{proof}
}


\section{Technical Lemmas}	\label{sec:technical}

\changes{This section states some technical results that will be used throughout this paper. Proofs are given at the end of the paper in \Cref{sec:missing,sec:bernoulli}. The first part is devoted to a recursion that appears in $\G$ in the context of edge probabilities as seen in \Cref{recur}. The second part is about convergence of a sequence of dependent Bernoulli random variables, wherein the dependence between consecutive terms in the sequence is governed by the recursion from the first part. The crucial convergence and concentration result is \Cref{bernoulli-main} which is fundamental to our analysis for proving the average vertex degree asserted in \Cref{avgdeg}. We also generalise this convergence in \Cref{bernoulli-main3}, which may be of independent interest, to a sequence where the consecutive terms satisfy a weaker assumption than the recursion.}

\subsection{Growth Rates in a Recurrence Formula}	\label{sec:recur}

Consider the following recursion formula 
\begin{equation}	\label{def:recur}
x_{n+1} \eq f_{a}(x_{n}) \quad \text{for } \ n \ge 1, \qquad \text{ where } \ f_{a} \sep x \mapsto x \left(1 - a\,x\right) , \quad \text{for } \  a\in (0,1],
\end{equation}
where it is assumed that $x_{1}\in(0,1)$. The function $f_{a}\colon x\mapsto x(1-ax)$ used to generate this recurrence has the following properties from elementary calculus.
	
\begin{lemma}	\label{obsfa}
$f_{a}$ is a concave quadratic that
		\begin{enumerate}
			\item is increasing on $(-\infty,\frac{1}{2a}]$ and decreasing on $(\frac{1}{2a},\infty)$,
			\item has a maximum value of $\frac{1}{4a}$,
			\item does not have any nonzero fixed point,
			\item satisfies $f_{a}(x) \in (0,1)$ for $x\in(0,1)$,
			\item satisfies $f_{a}(x) < x$ for $x \neq 0$,
			\item satisfies  $af_{a}(x) \eq f_{1}(ax)$,
			\item satisfies $f_a(x)< f_b(x)$ for $x > 0$ if $a>b$.
		\end{enumerate}
\end{lemma}
 
The terms in the recurrence generated by $f_{a}$ are lower and upper bounded by constant multiples of $1/n$. 

\begin{lemma}	\label{xi-bounds}
Recurrence~\eqref{def:recur} generates a decreasing sequence with each term $x_{n}$, for $n\ge 2$, bounded as
\begin{equation}	\label{piecewise}
\frac{1}{a}\ell_n(x_1) \:\le\: x_{n}  \:<\: \frac{1}{a} \,\min\left\{\frac{1}{4},\, \frac{1}{n- 1+\frac{1}{x_1}}\right\}, \quad \text{where } \
\ell_n(x_1) \define \begin{cases}
\frac{1}{(\sqrt{n}+1)^2}, & n \ge \left(\frac{1}{f_{1}(x_{1})} - 1 \right)^{2} \medskip\\
\frac{f_{1}(x_{1})}{n}, & \text{otherwise.}
\end{cases}
\end{equation}
\end{lemma}

Another result that we will need is lower and upper bounds on the sum of terms in the recurrence for large enough $n$. We show that the partial sums in the recurrence are $\Theta(\log{n})$.

\begin{lemma}	\label{boundsonsum} 
Denote $f_{1}:=f_{1}(x_{1})$ and $\eta(x_{1}) \define 2(1-f_{1})\log{f_{1}} - \left(\frac{5}{2}+\log{2} \right)f_{1}$. 
For $n\ge (\frac{1}{f_{1}} - 1)^{2}$, we have
	\[
ax_{1} + \eta(x_{1}) + \frac{2}{1+\sqrt{n}} + \log{n} \,\le\, a\sum_{i=1}^{n}x_{i} \,\le\, ax_{1} + \log{\left(1 + (n-1)x_{1} \right)} .
	\] 
\end{lemma}

\subsection{Convergence of a Bernoulli sequence}	\label{sec:markovchain}


Consider a sequence of Bernoulli r.v.'s $\{\X_{n}\}$, defined on the same probability space $(\Omega,\mathcal{F},\mathbb{P})$, that are not assumed to be iid. Denote  the marginal probabilities and partial sum by
\[
\pp_{n} \define \P{\X_{n} = 1} \quad n \ge 1, \qquad S_{n} \define \X_{1} + \X_{2} + \cdots + \X_{n} .
\]
We analyse the ratio of this partial sum to $\log{n}$, and so for convenience let us denote
\[
\YY_{n} \define \frac{S_{n}}{\log{n}} \, .
\]
Our main result of this section is that when the probability sequence $\{\pp_{n} \}$ satisfies the recurrence formula \eqref{def:recur} with $a$ equal to some constant $\p$, then $\YY_{n}$ concentrates to $1/\p$. Furthermore, we also establish that the mean of $\YY_{n}$, which is equal to $\E{S_{n}}/\log{n}$, converges to $1/\p$ and the distribution of $\YY_{n}$ grows at a quadratic rate in the limit.

\begin{theorem}	\label{bernoulli-main}
Suppose the sequence $\{\pp_{n}\}$ is such that there is some $\p\in(0,1]$ for which $\pp_{n} = f_{\p}(\pp_{n-1})$ for all $n\ge 2$. Then, $\E{\YY_{n}} \longrightarrow 1/\p$ and $\YY_{n} \overset{p}{\longrightarrow} 1/\p$. 
\ifthenelse{\isundefined{\ics}}{
Furthermore, the distribution of $\YY_{n}$ grows at a quadratic rate in the limit as follows, 
\[
\frac{1}{1-\theta} \:\ge\: \lim_{n\to\infty} \P{\frac{\YY_{n}}{\E{\YY_{n}}} \,\ge\, 1-\theta} \:\ge\: \theta^2 , \qquad \theta\in[0,1).
\]
}
{}
\end{theorem}

We discuss here some consequences of this theorem. First, we have that the partial sum is dominated by arbitrary powers of $n$. 

\begin{corollary}
Under the conditions of \Cref{bernoulli-main}, w.h.p. $S_{n} = O(n^{\epsilon})$ for every $\epsilon >0$.
\end{corollary}
\begin{proof}
Fix any $\epsilon > 0$. Note that $\log{n} = O(n^{\epsilon})$, and so there exists a constant $C_{\epsilon}$ and integer $N_{\epsilon}$ such that $\log{n} \le C_{\epsilon}n^{\epsilon}$ for all $n\ge N_{\epsilon}$. Therefore, for all $n\ge N_{\epsilon}$,
\begin{align*}
\P{\frac{S_{n}}{C_{\epsilon}n^{\epsilon}} \,\ge\, 1 + \frac{1}{\p}} \:\le\: \P{\frac{S_{n}}{C_{\epsilon}n^{\epsilon}} \,\ge\, \epsilon + \frac{1}{\p}} 
\:\le\: \P{\frac{S_{n}}{\log{n}} \,\ge\, \epsilon + \frac{1}{\p}}
\:\le\: \P{\abs{\frac{S_{n}}{\log{n}}  - \frac{1}{\p}} \,\ge\, \epsilon} ,
\end{align*}
where the first inequality is due to $\epsilon + 1/\p < 1 + 1/\p$, and the last inequality is due to $\abs{S_{n}/\log{n}  - 1/\p} \ge S_{n}/\log{n}  - 1/\p$. The concentration of $S_{n}/\log{n}$  to $1/\p$ in \Cref{bernoulli-main} means that the rightmost probability converges to 0, and then the above chain implies that the leftmost probability also goes to 0, which leads to $S_{n} =O(n^{\epsilon})$ w.h.p.
\end{proof}

The second consequence of our theorem is that it subsumes the case where the stochastic process $\{\X_{n}\}$ forms a non-homogenous Markov chain whose transition matrix obeys a decay property.

\begin{corollary}	\label{bernoulli-main2}
Suppose that $\{\X_{n}\}$ is a Markov chain of Bernoulli r.v.'s, so that  for every $n\ge 1$ we have
\[
\P{\X_{n+1}=y_{n+1}}[\X_{1}=y_{1},\dots,\X_{n}=y_{n}] \eq \P{\X_{n+1}=y_{n+1}}[\X_{n}=y_{n}],
\]
and also suppose that there is some $\p\in[0,1)$ such that
\[ 
\P{\X_{n+1}=1}[\X_{n} = 0] \eq  \P{\X_n=1}, \qquad
\P{\X_{n+1}=1}[\X_{n} = 1] \eq (1-\p) \, \P{\X_n=1} .
\]
Then, $\E{\YY_{n}} \longrightarrow 1/\p$ and $\YY_{n} \overset{p}{\longrightarrow} 1/\p$.
\end{corollary}
\begin{proof}
This is a special case of \Cref{bernoulli-main} because it is straightforward to verify that $\pp_{n} = f_{\p}(\pp_{n-1})$ (also see \Cref{recur}  for arguments in the context of the chain $\{X^{i}_{j}\}_{j < i}$ for each vertex $v_{i}$ in $\G$).
\end{proof}


A third consequence of our main theorem 
is that we show convergence of expectation for the sequence $\{\YY_{n} \}$ under a weaker assumption, namely, when the probability sequence $\{\pp_{n}\}$ is sandwiched between two sequences $\{f_{a_{n}}(\pp_{n-1}) \}$ and $\{f_{b_{n}}(\pp_{n-1}) \}$ that are generated using two converging sequences $\{a_{n}\}$ and $\{b_{n}\}$.

\begin{corollary}	\label{bernoulli-main3}
Suppose there exist two converging sequences $\{a_n\}, \{b_n\} \subset (0,1]$ having the same limit $\p > 0$ and such that the corresponding sequences $\{f_{a_{n}} \}$ and $\{f_{b_{n}} \}$ bound the marginal probabilities of $\{\X_{n}\}$ as $f_{a_{n}}(\pp_{n-1}) \,\le\, \pp_n \,\le\, f_{b_{n}}(\pp_{n-1})$ for all $n \ge 2$. 
Then, $\E{\YY_{n}} \longrightarrow 1/\p$.
\end{corollary}

Note that the weaker assumption made in this corollary also means that we do not guarantee convergence in probability, unlike as in \Cref{bernoulli-main2}.

Let us also make some remarks about the claims in \Cref{bernoulli-main}.


\begin{remark}	\label{convgrem1}
For a general random sequence, convergence of expected value and convergence in probability do not imply each other, but when the sequence is uniformly integrable then the latter implies the former \cite[cf.][]{stackex}. This indicates that if we can show that the sequence $\{\YY_{n}\}$ is uniformly integrable then establishing $\YY_{n} \overset{p}{\longrightarrow} 1/\p$ gives us $\E{\YY_{n}} \longrightarrow 1/\p$. However, we do not use this implication in our proof since we do not think that $\{\YY_{n}\}$ is uniformly integrable. 
Instead, our proof first establishes convergence of $\E{\YY_{n}}$ and uses it to argue convergence in probability. 
\end{remark}

\begin{remark}	\label{convgrem2}
For a random sequence $\{\XX_{n}\}$, there are two notions when talking about the expected values converging. The first is that of convergence \emph{of} expectation where there is another r.v. $X$ such that $\E{\XX_{n}} \longrightarrow \E{\XX}$ (if $\E{\XX_{n}}$ converges to a finite number $c$ then one could simply define $\XX$ to be equal to $c$  a.s.). The second  notion is that of convergence \emph{in} expectation where $\E{\abs{\XX_{n} - \XX}} \longrightarrow 0$ for some r.v. $\XX$. Markov's inequality gives us that the latter implies convergence in probability, and it is well-known that the reverse implication is true if and only if $\{\XX_{n}\}$ is uniformly integrable. Since we do not think that uniform integrability holds for our sequence $\{\YY_{n}\}$, 
we cannot use \Cref{bernoulli-main} to say that $\{\YY_{n}\}$ converges in expectation to $1/\p$.
\end{remark}




\section{Vertex Degrees}	\label{sec:lower}

Our main goal here is to prove \Cref{lower}. A key part of this proof is analyzing the average vertex degree in this graph. Let us formally define vertex degree and average degree in $\G$. The degree of vertex $v_i$ is a random variable denoting the number of edges incident on $v_i$, given by the expression
\[
\deg{v_i} \eq \sum_{j=1}^{i-1} X^{i}_{j} + \sum_{j=i+1}^n X^{j}_{i} .
\]
The average degree $d(\G)$ is the average vertex degree across the entire graph,
\[
d(\G) \eq \frac{1}{n} \sum_{i=1}^n \deg{v_i} .
\]

\subsection{Concentration of average degree}

Our main tool in lower bounding $\alpha(\G)$ is the following concentration result about average degree, which can also be equivalently stated as saying that $d(\G)/\log{n} \overset{p}{\longrightarrow} \revise{2/\rr} $. To establish this concentration, we use a key technical result about a dependent Bernoulli sequence; the analysis of this sequence is presented later in \Cref{sec:bernoulli}.

\begin{proof}[\textbf{Proof of \Cref{avgdeg}}]

\changes{Denote $S^{i}_{i-1} := \sum_{j=1}^{i-1}X^{i}_{j}$ for $i \ge 2$, and $S^{1}_{0} = 0$.} 
The sum of vertex degrees is
\[
\sum_{i=1}^n \deg{v_{i}} = \sum_{i=1}^n \left( \sum_{j=1}^{i-1}X^{i}_{j} + \sum_{k=i+1}^n X^{k}_{i}\right) = 2 \sum_{i=1}^{n} \sum_{j=1}^{i-1} X^{i}_{j} = 2 \sum_{i=1}^n S_{i-1}^{i} = 2\sum_{i=1}^{n-1} S_{i}^{i+1} ,
\]
\changes{which leads to the average degree of the graph becoming $d(\G) = \frac{2}{n} \sum_{i=1}^{n-1} S_{i}^{i+1}$. 

Now we invoke \Cref{markovchain} and its Markov chain $\{\mathcal{X}_{1},\ldots,\mathcal{X}_{n} \}$. From this lemma, we know that $S_{i} := \mathcal{X}_{1} + \cdots + \mathcal{X}_{i}$ has the same distribution as $S^{i+1}_{i}$.} We denote this by $S_{i}^{i+1}\sim S_{i}$. Also, since each random variable in the Markov process $\{\mathcal{X}_{i} \}_{i\ge 1}$ is nonnegative, we have $\mathbf{0} \preceq S_{1} \preceq S_{2} \preceq \cdots \preceq S_{n}$, where $\preceq$ denotes first-order stochastic dominance ($X \preceq Y$ means that $P\{X\ge c\}\le P\{Y\ge c\}$ for all real $c$) and where $\mathbf{0}$ denotes the random variable that takes value 0 with probability 1. Hence, $S^{i+1}_{i}\sim S_{i}$ implies that $ S^{i+1}_{i}\preceq S_{n}$ for all $i\le n$, which leads to
\begin{equation}	\label{ineq:graph-degree}
d(\G) = \frac{2}{n} \sum_{i=1}^{n-1} S_{i}^{i+1} \preceq \frac{2(n-1)}{n}S_{n} \preceq 2S_{n} .
\end{equation}

From \Cref{bernoulli-main2}, we have that for every $\epsilon > 0$,
\begin{equation}\label{eqn:sn}
\lim_{n\rightarrow \infty} \P{ \left| \frac{S_{n}}{\log n} - \frac{1}{\rr } \right| \le \frac{\epsilon}{\rr } } = 1
\end{equation}
and then \eqref{ineq:graph-degree} yields the following upper bound \changes{for every $\epsilon_{0}>0$},
\begin{equation}
	\label{ineq:degree-graph-to-sum}
	\P{ \frac{d(\G)}{\log n}  \ge \frac{2+\epsilon_0}{\rr } } 
	\le 
	\P{ \frac{2S_{n}}{\log n} \ge \frac{2+\epsilon_0}{\rr } }
	\le 
	\P{\left| \frac{S_{n}}{\log n} - \frac{1}{\rr } \right| \ge \frac{\epsilon_0}{2\rr }} .
\end{equation}
Moreover,
\[
\P{\left| \frac{S_{n}}{\log n} - \frac{1}{\rr } \right| \ge \frac{\epsilon_{0}}{2\rr }} 
\le 
1 - \P{\left| \frac{S_{n}}{\log n} - \frac{1}{\rr } \right| \le \frac{\epsilon_{0}}{3\rr }}. 
\]
If we pick $\epsilon:=\epsilon_{0}/3$ in \eqref{eqn:sn} and take limit of both sides, the right hand side of the above goes to 0 and so does the left.
Together with \eqref{ineq:degree-graph-to-sum}, we have for all $\epsilon_0 >0$,
\begin{equation}
\label{ineq:degree_upper_bound}
\lim_{n\rightarrow \infty} \P{ \frac{d(\G)}{\log n}  \ge \frac{2+\epsilon_0}{\rr }} = 0 \quad\implies\quad \lim_{n\rightarrow \infty} \P{ \frac{d(\G)}{\log n}  \le \frac{2(1+\epsilon_0)}{\rr } } = 1 .
\end{equation}
Therefore, $d(\G)$ is asymptotically bounded above by $2(1+\epsilon_0)/\rr  \log n$. For deriving the lower bound, we first mention that
for $n_c:=\lfloor cn\rfloor\in \mathbb{Z}_{\ge 0}$ with $c\in (0, 1)$,
\[
d(\G) = \frac{2}{n} \sum_{i=1}^{n-1}S_{i}^{i+1} \succeq \frac{2}{n}  \sum_{i=n_c}^{n-1}S_{i}^{i+1}  \succeq \frac{2(n-n_c)}{n} S_{n_c}\succeq 2(1-c)S_{n_c} .
\]
For the lower bound of $\frac{d(\G)}{\log n}$, let us now fix $\epsilon >0$. From \eqref{eqn:sn},
\[
\forall \epsilon > 0,\;\, \exists N_{\epsilon} \in \mathbb{Z}_{\ge 0} \; \text{ s.t. } \; \P{ \left| \frac{S_{n}}{\log n} - \frac{1}{\rr } \right| \le \frac{\epsilon}{2\rr} } \ge 1, \quad \forall n\ge N_{\epsilon}.
\]
Let 
\begin{equation}
	\label{eqn:c-define}
	c:= \frac{\epsilon-2\epsilon^{2}+2\epsilon^3}{2-\epsilon}
\end{equation}
Note that $c\in (0,1)$ is well-defined when $\epsilon \in (0, 1)$. In addition, $c$ is monotone increasing in terms of $\epsilon$. Define 
\[
M_\epsilon := \max\left\{\exp\left\{-\frac{1+\epsilon^{2}}{\epsilon^{2}}\log c  \right\}, \frac{N_{\epsilon}}{c} \right\} .
\]
If $n\ge M_{\epsilon}$, then $n\ge N_\epsilon/c$ and so
$n_{c}= \lfloor cn\rfloor \ge  N_{\epsilon}$. Moreover, 
\begin{equation}
\label{eqn:log_nc_ratio}
\frac{\log n}{\log n_{c}} = \frac{\log n}{\log n + \log c - 1} \le 1+\epsilon^2 .
\end{equation}
From $d(\G) \succeq 2(1-c)S_{n_c}$, 
\[
\P{ \frac{d(\G)}{\log n}  \ge \frac{2(1-\epsilon)}{\rr } } 
\ge
\P{ \frac{2(1-c)S_{n_{c}}}{\log n}  \ge \frac{2(1-\epsilon)}{\rr } } 
\ge
\P{ \frac{S_{n_{c}}}{\log n_{c}} \ge \frac{(1+\epsilon^{2})(1-\epsilon)}{(1-c)\rr } } 
\]
where the last inequality comes from $(1+\epsilon^2) S_{n_c}/\log n \ge S_{n_c}/\log n_c $ which is a straightforward result of \eqref{eqn:log_nc_ratio}. Substitute \eqref{eqn:c-define} and simplify the above term, we have
\[
\begin{aligned}
\P{ \frac{S_{n_{c}}}{\log n_{c}}  \ge \frac{(1+\epsilon^{2})(1-\epsilon)}{(1-c)\rr } }  
&= 
\P{  \frac{S_{n_{c}}}{\log n_{c}}  \ge \frac{2-\epsilon}{2\rr } } \\
&\ge
\P{ \left|  \frac{S_{n_{c}}}{\log n_{c}} - \frac{1}{\rr } \right| \le \frac{\epsilon}{2\rr } } 
\ge 1 - \epsilon .
\end{aligned}
\]
So far, we have shown that, for all $\epsilon >0$, there exists $M_{\epsilon}\in \mathbb{Z}_{\ge 0}$ such that 
\[
\P{ \frac{d(\G)}{\log n}  \ge \frac{2(1-\epsilon)}{\rr } } 
\ge 1-\epsilon, \quad \forall n\ge M_{\epsilon}.
\]
Consequently, $\frac{d(\G)}{\log n}  \ge  \frac{2(1-\epsilon)}{\rr }$ w.h.p.. Together with the upper bound, we obtain
\[
\lim_{n\rightarrow \infty} \P{\left|\frac{d(\G)}{\log n} - \frac{2}{\rr} \right|\le \frac{\epsilon}{\rr} } = 1, \quad \forall \epsilon > 0,
\]
which concludes our proof.
\end{proof}

\subsection{Establishing the lower bounds}

We will use the following lower bound on stability number of general graphs that arises from applying Tur\'an's theorem, which says that any graph not having a $K_{t+1}$ subgraph cannot have more edges than the Tur\'an graph corresponding to $n$ and $t$.

\begin{lemma}[{\cite[][Theorems 10 and 11]{coffman2003bounds}}]	\label{avgdeglem}
For any $n$-vertex graph $G$ having $m$ edges, the stability number $\alpha(G)$ can be lower bounded as follows,
\[
\alpha(G)\ge \max\left\{\frac{n^{2}}{n + 2m},\, \frac{2n-m}{3} \right\}.
\]
\end{lemma}

The first bound is equivalent to $n/(1+d(G))$, where $d(G) = 2m/n$ is the average degree of $G$, and this bound has been noted before, for e.g. by \citet{griggs1983lower}. It is dominated by the second bound if and only if $n/2 < m < n$. 

We are now ready to argue our first main lower bound.

\begin{proof}[\textbf{Proof of \Cref{lower}}]
For fixed $\epsilon > \epsilon_{0} > 0$, we observe if $n$ is sufficiently large then
\begin{equation}	\label{thm1eq}
\P{\alpha(\G) \ge \frac{n}{\frac{2+\epsilon}{\rr }\log n} } \ge \P{\alpha(\G) \ge \frac{n}{\frac{2+\epsilon_{0}}{\rr } \log n + 1} } \ge \P{\frac{d(\G)}{\log n} \le \frac{2+\epsilon_{0}}{\rr } } ,
\end{equation}
where the first inequality follows from $\frac{2+\epsilon}{\rr }\log n > \frac{2+\epsilon_{0}}{\rr } \log n + 1$ for $n > e^{\frac{\rr}{\epsilon - \epsilon_{0}}}$, and the second inequality has used the first bound in \Cref{avgdeglem} which tells us that $\alpha(\G) \ge n/(1 + d(\G))$. After taking the limit throughout in \eqref{thm1eq}, the right-most probability goes to 1 by \Cref{avgdeg}, and so the left-most term also has limit 1, which yields the lower bound of $\Omega(n/\log{n})$ on $\alpha(\G)$.
\end{proof} 


\revise{The above derived lower bound can also be stated in terms of the prime-counting function $\pi(n)$, which is defined as the number of primes less than or equal to $n$.

\begin{corollary}	\label{primes}
For every $\epsilon > 0$, we have w.h.p. that $\alpha(\G) \ge \dfrac{\rr }{2(2\epsilon+1)}\,\pi(n)$.
\end{corollary}
\begin{proof}
The well-known Prime Number Theorem  says that $\lim_{n\to\infty}\frac{\pi(n)}{n/\log{n}} \eq 1$. 
For every $\epsilon > 0$, replacing $n/\log{n}$ in the leftmost event in \eqref{thm1eq} with its lower bound $\pi(n)/(1+\epsilon)$ implies that w.h.p.
\[
\alpha(\G) \ge \frac{\rr}{2+\epsilon}\,\frac{\pi(n)}{1+\epsilon} .
\]
Since the denominator $\epsilon^{2} + 3\epsilon + 2$ is less than $4\epsilon+2$, our proof is complete.
\end{proof}
}


\section{Estimating Edge Probabilities} \label{sec:edge-prob}



\changes{We establish bounds on the different probabilities that arise in our analysis in the next two sections on the greedy algorithm and an upper bound on $\alpha(\G)$. The most crucial result, which comes at the end of this section in \Cref{Prob-Ak-is-Indep}, is an upper bound on the probability for a subset of vertices to be pairwise disconnected. Towards this end, we establish several lemmas, starting with the basic question of bounding the probability for an edge to exist in the random graph. Recall the notation from \Cref{sec:chain}.}

\begin{lemma}	\label{edgeprob}
For every $2 \le j \le i-1$, we have
\[
\frac{\ell_{j}\left(\rr p\right)}{\rr }\:\le\: \pp^{i}_{j} \:<\: \min\left\{\pp^{i}_{j-1},\, \frac{1}{\rr j}\right\} .
\]
\end{lemma}
\begin{proof}
By \Cref{recur}, for every $i$, the sequence $\{\pp^{i}_{1},\pp^{i}_{2},\dots,\pp^{i}_{i-1}\}$ obeys the recursion $x_{k+1} = x_{k}\left(1 - \rr x_{k}\right)$ starting with $x_{1} = \pp^{i}_{1} = p$. This recursion is the same as the one analysed earlier in \eqref{def:recur} and 
therefore, the claimed bounds follow from applying \Cref{xi-bounds} with $a=\rr$, and relaxing the upper bound of $1/\rr (j-1+1/p)$ to $1/\rr j$ because $1/p \ge 1$.
\end{proof}

\revise{The transition matrix in \cref{def:Pik}} gives conditional probabilities for consecutive \revise{Bernoulli} variables \revise{$X^{i}_{j-1}$ and $X^{i}_{j}$}. For two arbitrary variables \revise{$X^{i}_{j}$ and $X^{i}_{k}$ corresponding to any vertex $v_{i}$}, the conditional probability can be bounded as follows.

\begin{lemma}	\label{ineqOfCondProb}
For $1\le k< j< i$,
\[
 1-\pp^{i}_{j-1} \le \P{X^i_j=0}[X^i_{k}=0] \le \rr \pp^{i}_{j-1}, \qquad \r\pp^{i}_{j-1} \le \P{X^i_j=1}[X^i_{k}=1] \le \pp^{i}_{j-1} .
\]
\end{lemma}
\begin{proof}
Fix any $i\ge 2$. For ease of readability, we suppress the superscript $i$ in the probabilities. The Markovian property gives us $\P{X^i_j=0}[X^i_{k}=0] \eq \left(\mathbf{P}^{\revise{i}}_{k}\mathbf{P}^{\revise{i}}_{k+1}\cdots \mathbf{P}^{\revise{i}}_{j-1}\right)_{11}$, \revise{where the $P^{i}_{k}$ matrix defined in \cref{def:Pik} is used}. Since $\mathbf{P}^{\revise{i}}_{k}$ is a stochastic matrix for all $k$, the product matrix $\mathbf{M} := \mathbf{P}^{\revise{i}}_{k}\mathbf{P}^{\revise{i}}_{k+1}\cdots \mathbf{P}^{\revise{i}}_{j-1}$ is also stochastic, which implies that $\left(\mathbf{M} \, \mathbf{P}^{\revise{i}}_{j-1}\right)_{11} \eq (1-\pp^{\revise{i}}_{j-1})\revise{\mathbf{M}_{11}} + (\rr \pp^{\revise{i}}_{j-1})(1-\revise{\mathbf{M}_{11}})$ , thereby leading to $1-\pp^{\revise{i}}_{j-1}\le \left(\mathbf{M} \, \mathbf{P}^{\revise{i}}_{j-1}\right)_{11} \le \rr \pp^{\revise{i}}_{j-1}$. A similar argument gives bounds on the other conditional probability after observing that $\P{X^i_j=1}[X^i_{k}=1] = \left(\mathbf{M} \, \mathbf{P}^{\revise{i}}_{j-1}\right)_{22}$.
\end{proof}

We now use these bounds on edge probabilities to estimate the probability that a subset of vertices is independent. The first step towards this is to recognise that this probability can be expressed as a product of probabilities for a vertex to be disconnected from a subset.

\begin{lemma}	\label{probproduct}
For any subset $S := \{v_{j_1},v_{j_2},\ldots, v_{j_k} \}$ of $k$ vertices in $\G$, we have that
\[	
\P{S \text{ is independent}} \eq \prod_{i=2}^k\P{v_{j_i} \text{ is disconnected from }\{v_{j_1}, \ldots, v_{j_{i-1}}\}} .
\]
\end{lemma}
\begin{proof}
Wlog, assume the vertices in $S$ are sorted such that $j_1<j_2<\cdots <j_k$. 
Since a subset is independent if and only if it is pairwise disconnected, it is clear that the desired probability is equal to the joint probability $\P{X_{j_i}^{j_r}= 0 \colon r > i}$. Partition these edge variables based on their upper index into disjoint sets $B_{j_1}, B_{j_2},\ldots, B_{j_k}$ where $B_{j_r} = \{X_{j_1}^{j_r}, X_{j_2}^{j_r}, \ldots, X_{j_{r-1}}^{j_r}\}$. Note that $B_{j_1} = \emptyset$. From \Cref{iid,markovchain}, two different edge variables are dependent if and only if they have the same superscript and their subscripts are two consecutive numbers. Therefore, all variables inside $B_{j_r}$ are independent from these in $B_{j_l}$ with $l\neq r$ because their upper indices $j_r$ and $j_l$ are not the same. This independence implies that the probability can be written as the product (over $i=2,\dots,k$) of the probability that all variables in $B_{j_{i}}$ are equal to 0. The event that all variables in $B_{j_i}$ equal to 0 is equivalent to saying that $v_{j_i}$ is disconnected from $\{v_{j_1}, \ldots, v_{j_{i-1}}\}$.
\end{proof}


Now we use bounds on edge probabilities to bound the probability of disconnection for a vertex from a subset. Our bounds are independent from the indices of vertices in the subset and are solely determined by the size of the subset.

\begin{lemma}	\label{AddingNewVertex}
For any subset $A \subseteq \{v_1,\ldots, v_k\}$ of size $m$, for all $i\ge k$, 
	\[
	(1-p)\prod_{j=1}^{m-1} (1-\pp^{i}_j) \:\le\: \P{v_i \text{ is disconnected from }A} \:\le\:  \prod_{j=i-m}^{i-1} (\rr  \pp^{i}_j) .
	\]
\end{lemma}
\begin{proof}
Fix $i \ge k$ and let $D_i(m)$ denote the event that $v_i$ is disconnected from $A$. 
We argue the lower bound first.
		\[
			\P{D_i(m)} = \P{X^i_{j_1}=0, \ldots, X^i_{j_m}=0}=\P{X^i_{j_1}=0}\prod_{k=2}^{m}\P{X^i_{j_{k}}=0}[X^i_{j_{k-1}}=0] .
		\]
By \Cref{ineqOfCondProb}, $\P{X^i_{j_{k}}=0}[X^i_{j_{k-1}}=0]$ is bounded below by $1-\pp_{j_{k}-1}$, where we suppress the superscript $i$ on probabilities for ease of readability. Therefore,
		\[
			\P{D_i(m)}
			\ge   \P{X^i_{j_1}=0}\prod_{k=2}^m (1-\pp_{j_{k}-1})
			= (1-\pp_{j_1})\prod_{k=2}^m  (1-\pp_{j_{k}-1})
		\]
We know from \Cref{edgeprob} that the sequence $\{\pp_{j}\}$ is decreasing, and hence the sequence $\{1 - \pp_{j}\}$ is increasing. Immediately, we have $1-\pp_{j_1}\ge 1-p$ since $p=\pp_1$ is the largest value among all $\pp_j$. Note that $\prod_{k=2}^m  (1-\pp_{j_{k-1}})$ contains $m-1$ distinct values $\pp_{j_1-1}, \pp_{j_2-1},\ldots,\pp_{j_{m}-1}$ where $j_{m}> \cdots > j_1\ge 1$. Clearly, $j_{k}\ge k$ for all $1\le k\le m-1$. Therefore, $\pp_{j_{k}-1} \le \pp_{k-1}$ for all $i$ and then $\prod_{k=2}^m  (1-\pp_{j_{k}-1}) \ge \prod_{j=1}^{m-1}  (1-\pp_{j})$. Consequently,
		$(1-\pp_{j_1})\prod_{k=2}^m  (1-\pp_{j_{k-1}})
		\ge (1-p)\prod_{j=1}^{m-1} (1-\pp_j)$ which is our desired lower bound.
		
The upper bound can be obtained similarly.
		\[
		\begin{aligned}
			\P{D_i(m)} \eq \P{X^i_{j_1}=0, \ldots, X^i_{j_m}=0} 
			&\eq \P{X^i_{j_1}=0}\prod_{k=2}^{m}\P{X^i_{j_{k}}=0}[ X^i_{j_{k-1}}=0]\\
			&\le \P{X^i_{j_1}=0} \prod_{k=2}^m(\rr \pp_{j_{k}-1})\\
			&\le \prod_{j=i-m}^{i-1}(\rr  \pp_j) .
		\end{aligned}
		\]
The first inequality is obtained using \Cref{ineqOfCondProb}. The last inequality is because $\P{X^i_{j_1}=0} \le 1$ and we claim that the product term is upper bounded by $\prod_{j=i-m}^{i-1}(\rr  \pp_j)$. To argue this claim, we use the assumption that $A$ only contains vertices with index less than $i$. Hence, $j_{1} < j_{2} <\cdots<j_{m} \le i-1$ and $j_{k}\le i-1-(m-k)$ for all $1\le k\le m-1$, which implies that $\pp_{j_{k}-1}\ge \pp_{i-2-m+k}\ge \pp_{i-1-m+k}$. Therefore, $\prod_{k=2}^m(\rr \pp_{j_k-1})\le \prod_{j=i-m}^{i-1}(\rr  \pp_j)$. 		
\end{proof}

%

Putting together the previous two lemmas leads to the following upper bound for an independent set.

\begin{proposition}	\label{Prob-Ak-is-Indep}
The probability $P_{k}$ that a subset of $k$ vertices in $\G$ is an independent set satisfies the following bound when $n-k+1\ge \left( 1/\tau -1 \right)^2$, for $\tau := f_{1}\left(\rr p\right) = \rr p \left(1 - \rr p \right)$,
\[
P_{k} \:\le\: \prod_{i=1}^{k-1} \left(1 \,-\, \frac{\r}{\rr }\frac{1}{\left(\sqrt{n-i}+1\right)^2}\right)^i .
\]
\end{proposition}
\begin{proof}
We use the expression from \Cref{probproduct}. For each term in the product, applying the upper bound from \Cref{AddingNewVertex} relaxes it as follows,
\begin{equation}
	\label{ineq:indep-set-ak-vji}
\P{v_{j_i} \text{ is disconnected from }\{v_{j_1}, \ldots, v_{j_{i-1}}\}} \:\le\:  \prod_{k=j_i-(i-1)}^{j_i-1} (\rr \pp_k)	,
\end{equation}
where for the sake of brevity, we have suppressed the superscript $j_{i}$ on the probabilities. Since $j_1<j_2<\cdots<j_k\le n$, we have $j_i \le n-(k-i)$. \Cref{edgeprob} tells us that $\{\pp_k\}$ is a decreasing sequence. Therefore, $\pp_{j_{i}-t} \ge \pp_{n-(k-i)-t}$ for all $t\in \{1,2,\ldots,i-1\}$ and now we can further relax \eqref{ineq:indep-set-ak-vji} as
\[
\P{v_{j_i} \text{ is disconnected from }\{v_{j_1}, \ldots, v_{j_{i-1}}\}}\le \prod_{k=j_i-(i-1)}^{j_i-1} (\rr \pp_k)	\le \prod_{t=k-(i-1)}^{k-1} (\rr \pp_{n-t})	
\]
Applying this inequality to the expression from \Cref{probproduct},
\[
P_{k} \:\le\: \prod_{i=2}^k  \prod_{t=k-i+1}^{k-1} (\rr \pp_{n-t}) \eq \prod_{i=1}^{k-1} \left(1- \r \pp_{n-i}\right)^i .
\]
We use the lower bound $\pp_{n-i} \ge \psi_{p,\r}(n-i)/\rr $ obtained from  \Cref{edgeprob} to eliminate $\pp_{n-i}$.
\[
\P{subset \text{ is independent} } 
\le
\prod_{i=1}^{k-1} \left(1- \r q_{n-i}\right)^i
\le 
\prod_{i=1}^{k-1} \left(1- \frac{\r}{\rr }\psi_{p,\r}(n-i)\right)^i
\]
Since $\psi_{p,\r}(n-i):=\ell_{n-i}(\rr p)$, and $\ell_n(x) = 1/(\sqrt{n}+1)^2$ if $n\ge 1/(f_1(x)+1)^2$,  we obtain an explicit upper bound for sufficiently large $n-k$.
\end{proof}

\begin{remark}
It is natural to ask about a lower bound on the probability $P_{k}$ from \Cref{Prob-Ak-is-Indep}. Using the expression from \Cref{probproduct}, the lower bound from \Cref{AddingNewVertex}, and the upper bound on each $p^{i}_{j}$ from \Cref{edgeprob} yields a lower bound of a higher power of $1-p$, and so we do not obtain a very good estimate that goes to 1 for high values of $n$.
\end{remark}

\section{Greedy algorithm}	\label{sec:greedy}

Now we analyse the performance of a greedy algorithm for approximating (lower bounding) $\alpha(\G)$. 
Our algorithm sorts vertices in the order $v_{1}$ to $v_{n}$, and starting with an empty set iteratively builds an independent set in $n$ iterations where at iteration $i$ the vertex $v_{i}$ is added to the independent set if this addition preserves the new set to also be independent. The definition of our random graph is such that one can intuitively see that the vertex degrees are in the descending order from $v_{1}$ to $v_{n}$ w.h.p. Hence, our simple greedy algorithm can be regarded as the opposite of what one would use as a greedy algorithm for approximating $\alpha(\G)$ wherein the vertices would be sorted in increasing order of their degrees ($v_{n}$ to $v_{1}$ for $\G$). Thus, lower bounding the performance of our algorithm tells us (intuitively speaking) the worst we can do in any greedy algorithm for approximating $\alpha(\G)$. We prove here the asymptotic lower bound guarantee that was stated earlier in \Cref{sec:results}. 

\begin{proof}[\textbf{Proof of \Cref{greedy}}]
We adopt a similar strategy as \citet{grimmett1975colouring} for $G_{n,p}$. Let $\Stab$ be the independent set created during the iterations of the greedy algorithm. We analyze the number of iterations needed to increase the size of $\Stab$ by 1. Denote $R_{m+1}$ to be the number of iterations needed to add a new vertex into $\Stab$ since when $|\Stab|=m$. Observe that $R_{m+1}$ follows a geometric distribution whose parameter is $\P{v_i \text{ is disconneted from }\Stab}$. Suppose $\Stab = \{v_{j_1},\ldots, v_{j_m}\}$, then
\[
\E{R_{m+1}} = \frac{1}{\P{ v_i \text{ is disconnected from }\Stab}}, \quad \forall i>j_m
\]
We need to estimate the probability of $v_i$ can be added into $\Stab$ i.e. $v_i$ is disconnected from $\Stab$. By \Cref{AddingNewVertex},
\[
\frac{1}{\prod_{j=i-m}^{i-1} (\rr \pp_j)} \le \E{R_{m+1}} \le  \frac{1}{(1-p)\prod_{j=1}^{m-1} (1-\pp_j)}
\]
Note that there exists a positive integer $w$ such that $\pp_i\le w/i$ for all $i$. More precisely, we can take $\tau = \lceil 1/\rr  \rceil$ by \Cref{edgeprob}. Denote $\gamma = (1-p)\prod_{j=1}^\tau (1-\pp_j)$. Then 
we can further relax the upper bound above when $m > \tau$.
\[
\E{R_{m+1}} \le \frac{1}{\gamma\prod_{j=\tau+1}^m (1-\frac{\tau}{j})} = \frac{1}{\gamma} \frac{m!}{(m-\tau)!\tau!} =\frac{1}{\gamma}\binom{m}{\tau}\le \frac{1}{\gamma}\left(\frac{me}{\tau}\right)^\tau .
\]
If $k$ is large enough, then
\[
\E{\sum_{m=1}^{k}R_m} \le \sum_{m=1}^\tau \E{R_m}+\sum_{m=\tau+1}^k \frac{1}{\gamma}\left(\frac{me}{\tau}\right)^\tau = O\left( k^{\tau+1} \right)
\]

Note that $\sum_{m=1}^\tau \E{R_m}$ is a constant and $\sum_{m=\tau+1}^k \frac{1}{\gamma}\left(\frac{me}{\tau}\right)^\tau$ is bounded above by $\frac{k}{\gamma}\left(\frac{ke}{\tau}\right)^\tau = O\left( k^{\tau+1} \right)$. For the variance, since $R_m$ follows geometric distribution,
\[
\begin{aligned}
	\operatorname{Var}{\sum_{m=1}^k R_m} &\le \sum_{m=1}^k \E{R_m^2} \le 2\sum_{m=1}^k \E{R_m}^2\\
	&\le\sum_{m=1}^w \E{R_m}+\sum_{m=d+1}^k\frac{1}{\gamma^2}\left(\frac{m}{\tau}\right)^{2\tau}\\ 
	&=  O\left( k^{2\tau+1} \right)
\end{aligned}
\]
For convenience of notation, we denote $\sum_{m=1}^k R_m$ as $Z_k$.
Let 
\[
k: =k(n) := \left(\left(\frac{\tau}{e}\right)^{\tau} \cdot \frac{\gamma}{3}\cdot n\right)^{\frac{1}{\tau+1}} = O\left(n^{\frac{1}{\tau+1}}\right)
\] 
Then, we have $\E{Z_k} = O(n)$ and $\operatorname{Var}{Z_k} =O\left(n^{\frac{2\tau+1}{w+1}} \right)$. In particular, $\E{Z_k} < \frac{n}{2}$ for sufficiently large $n$, which implies that
\[
\P{Z_k\ge n} \le \P{Z_k\ge\E{Z_k}+ \frac{n}{2} }\le \P{\left|Z_k-\E{Z_k}\right|\ge \frac{n}{2} }.
\]
Therefore, when $n\rightarrow \infty$, we get that
\[
\P{Z_k\ge n} \le \P{\left|Z_k-\E{Z_k}\right|\ge \frac{n}{2} }\le \frac{4\operatorname{Var}{Z_k}}{n^2} = O\left( n^{-\frac{1}{\tau+1}} \right).
\]
Hence, $Z_k$, the number of iterations needed to obtain a independent set of size $k$, is at most $n$ w.h.p. when $n\rightarrow\infty$ and $k=n^{1/(\tau+1)}$. Conversely, if our algorithm has $n$ iterations, then the size of independent set is at least $k$ w.h.p., i.e. the output of our greedy method is $\Omega\left(n^{1/(\tau+1)}\right)$. 
\end{proof}

\changes{
\Cref{greedy} gives a weaker lower bound on $\alpha(\G)$ than that from \Cref{lower}. 
Since the construction of our random graph is such that one can intuitively see that the vertex degrees are in the descending order from $v_{1}$ to $v_{n}$ w.h.p., the greedy algorithm we consider gives a worst-case performance bound on the greedy algorithm that one typically uses to approximate $\alpha(\G)$ by sorting vertices from smallest to largest degree (in our case that would be $v_{n}$ to $v_{1}$). This is noted formally in the next result which compares the performance ratios of the two methods. The performance ratio of an algorithm $A$ that outputs an independent set of size $\alpha_{A}(G)$ for any graph $G$ (and hence provides a lower bound on $\alpha(G)$) is denoted by $\rho_{A}(n)$ and defined as the maximum ratio $\alpha(G)/\alpha_{A}(G)$ over all $n$-vertex graphs $G$. Let $\greedy$ denote our greedy algorithm which iteratively adds vertices in the order $v_{1}$ to $v_{n}$ and $\greedy[min]$ denote the algorithm which picks the smallest degree vertex at each iteration.

\begin{corollary}	\label{greedyratio}
W.h.p., $\rho_{\greedy} = O(n^{\frac{1}{1+\rr}})$ and $\rho_{\greedy[min]} \,\revise{\le}\, 1 + \left(\frac{1}{\rr}+\epsilon \right)\log{n}$ for $\epsilon > 0$.
\end{corollary}
\begin{proof}
The first bound is directly from combining \Cref{greedy} and \Cref{upper}. The second bound uses the general result of \citet[Corollary~4][]{Halldorsson1997} which says that $\rho_{\greedy[min]} \le (d(G)+2)/2$, and the upper bound on $d(\G)$ from \Cref{avgdeg}.
\end{proof}
}
 



\section{Upper Bound}	\label{sec:upper}

Our goal here is to prove the upper bound on $\alpha(\G)$ \revise{in} \Cref{upper}. We use the first moment argument, which has also been used for the \revise[2]{Erd\H{o}s-R\'enyi-Gilbert} random graph model \cite{matula1976largest,frieze1990independence,Frieze2016}. More precisely, let $H_{k,n}$ be the number of independent sets of order $k$ in $\G$. We want to find the minimum $k$ such that $\P{H_{k,n}>0}\to 0$ as $n\to\infty$.  The first step towards this is to estimate the probability that a subset of vertices forms an independent set.

\subsection{Proof of Theorem~\ref{upper}}



The function
\[
\revise{\varphi_{\r}}(c) \define c\, (1-\log c) + \frac{\r}{\rr } \, \left(c+\log(1-c)\right) \eq \frac{c}{\rr }  - c \log{c} + \frac{\r}{\rr }\log{(1-c)}.
\]
appears in our analysis and we will need to use negative values of this function over the positive part of its domain, which lies within $(0,1)$. For this purpose we claim that there is a unique positive root beyond which $\revise{\varphi_{\r}}$ is negative-valued, and although it is difficult to characterise this root explicitly, we provide an upper bound on it in terms of $\r$.

\begin{claim}	\label{rootf}
For every $\r \in (0,1)$ the function $\revise{\varphi_{\r}}$ has a unique positive root $c^{*}$ and $\revise{\varphi_{\r}}(c) <0$ for all $c\in (c^\ast, 1)$. Furthermore, $c^{*} < e^{-\r} + 0.1\r$.
\end{claim}


We argue an upper bound on $\alpha(\G)$ that is arbitrarily close to the root $c^{\ast}$ of the function $\revise{\varphi_{\r}}$. In particular, we prove that for every $\epsilon > 0$, w.h.p. $\alpha(\G) \le \left(c^{\ast}+\epsilon\right)\,n$. Since \Cref{rootf} tells us that $c^{\ast} < e^{-\r} + 0.1\r$, it follows that  $\alpha(\G) \le \left(e^{-\r} + 0.1\r\right)n$.

Let $H_{k,n}$ be the number of independent sets of order $k$ in $\G$. We want to find a minimum $k:=k(n)$ such that $\P{H_{k,n}>0}\to 0$ as $n\to\infty$.  Then, with such a $k(n)$, we have w.h.p. $\alpha(\G)\le k(n)$. Observe that $\E{H_{k,n}} = \sum_{i=1}^n i\P{H_{k,n}=i} \ge \sum_{i=1}^n \P{H_{k,n}=i} = \P{H_{k,n} > 0}$. It is sufficient to show that $\E{H_{k,n}}\to 0$ w.h.p. 

Observe that $\E{H_{k,n}}$ can be written as the sum of probabilities of all possible subsets of size $k$ to be independent set, 
\[
	\E{H_{k,n}} \eq  \sum_{\substack{A\subseteq \{v_1,v_2,\ldots, v_n\}\colon\\|A|=k}} \P{A \text{ is independent set}} 
\]
Using the upper bound from \Cref{Prob-Ak-is-Indep} for the probability in the summand, it follows that when $n-k$ is large enough, 
\begin{equation}	\label{exp-bound}
\E{H_{k,n}} \: \le \: \binom{n}{k} \,  \prod_{i=1}^{k-1} \left(1-\frac{\r}{\rr } \frac{1}{(\sqrt{\revise{n}-i}+1)^2}\right)^i .
\end{equation}

Let $k$ be an upper bound on $\alpha(\G)$. Since $\G$ is a $n$-vertex graph, we can assume that $k = c\, n$ for some constant $c$ which is to be determined. We claim that 

\begin{claim}	\label{subset-bound}
We have
\[
\prod_{i=1}^{k} \left(1- \frac{h}{(\sqrt{\revise{n}-i}+1)^2}\right)^i \le \Exp{ h \,  (\log(1-c)+c)\, n +O(n)} .
\]
\end{claim}

Before proving this claim, let us argue why it finishes our proof of this theorem. When $k$ is linear in terms of $n$, it is well-known that binomial coefficient $\binom{n}{k}$ is bounded above by $(ne/k)^k$. \revise{\Cref{subset-bound}} applied to \eqref{exp-bound} implies that 
\[
\E{H_{k,n}} \:\le\: \left(\frac{ne}{k}\right)^k\: \prod_{i=1}^{k-1}  \left(1-\frac{\r}{\rr } \frac{1}{(\sqrt{\revise{n}-i}+1)^2}\right)^i \eq \Exp{n \, \revise{\varphi_{\r}}(c)  + O(n)} .
\]
If $\revise{\varphi_{\r}}(c) < 0$, then $\lim_{n\to \infty}\E{H_{k,n}} = 0$ which is what we want for $k=cn$ to be an upper bound on the stability number. For the tightest bound, we take the smallest value of $c$ with $\revise{\varphi_{\r}}(c) < 0$. \Cref{rootf} shows that $\revise{\varphi_{\r}}(c) < 0$ for all $c\in(c^{*},1)$ where $c^{*}$ is the unique positive root of $\revise{\varphi_{\r}}$. Therefore, taking $c = c^{\ast} + \epsilon$ yields the desired upper bound of $c^{\ast}+\epsilon$.

It remains to prove the two claims used in the above proof. 

\subsubsection{Proof of Claim~\ref{rootf}}

It is clear that the domain of $\revise{\varphi_{\r}}$ is in $(-\infty,1)$. Let us first prove uniqueness of the root in $(0,1)$. We have
	\[
\frac{\partial}{\partial c}\revise{\varphi_{\r}}(c) \eq  \frac{\r}{\rr }\left( 1-\frac{1}{1-c}\right)-\log c, \qquad
\frac{\partial^2}{\partial c^2}\revise{\varphi_{\r}}(c) \eq -\frac{1}{c}-\frac{\r}{\rr }\cdot \frac{1}{(1-c)^2} \,.
	\]
	As the second derivative $\frac{\partial^2}{\partial c^2}\revise{\varphi_{\r}}(c)$ is always negative, $\revise{\varphi_{\r}}(c)$ is concave for fixed $\r$. Therefore, there are at most $2$ roots. Next, we observe that
	\[
	\lim_{c\rightarrow 0} \revise{\varphi_{\r}}(c) = 0,\quad \lim_{c\rightarrow 0}f'_{\r}( c)=\infty   , \quad \lim_{c\rightarrow 1} \revise{\varphi_{\r}}(c) = -\infty
	\]
	The boundary value of $\revise{\varphi_{\r}}(c)$ at $c=0$ is 0. However, $c=0$ is not in the domain of $\revise{\varphi_{\r}}(c)$ and so there is at most 1 root for $\revise{\varphi_{\r}}(c)$. We now need to show such root actually exists between 0 and 1. 
	Notice that the slope near $c=0$ is positive and  there must exist $a(\r)\in (0,1)$ such that $\revise{\varphi_{\r}}(a(\r)) > 0$ for all $\r\in (0,1)$. Since the either boundary value at $c=1$ is $-\infty$, the continuity of $\revise{\varphi_{\r}}(c)$ implies the existence of a root $c^\ast\in (a(\r), 1)$. In particular, $\revise{\varphi_{\r}}(c) <0$ for all $c\in (c^\ast, 1)$ and $\revise{\varphi_{\r}}(c) > 0$ for all $c\in (0, c^\ast)$.
	
Although there is no explicit formula for the root $c^\ast$, we can find some upper bound function $c(\r)$ such that $1 > c(\r) > c^\ast$ for all $\r\in (0,1)$. 
Consider $c(\r) = e^{-\r} + \frac{\r}{10}$ and denote the function $g(\r)$ by
	\[
	g(\r) \define \revise{\varphi_{\r}}\left(c(\r)\right) \eq \left(1-\frac{1}{1-e^{-\r} - \frac{\r}{10}}\right)\, \r \:-\: \log\left(e^{-\r} + \frac{\r}{10}\right) .
	\]
We need to show $g(\r) <0$ for all $\r\in (0,1)$. It is clear that the boundary values of $g(\r)$ are $g(0) = 0$ and $g(1) \approx -0.1197$. Since $g(\r)$ is continuous and both boundary values are non-positive, it is sufficient to prove $g(\r)$ has no root within interval $(0,1)$. Suppose not, then $\exists\, \r^\ast \in (0,1)$ such that 
	\[
	\left(1-\frac{1}{1-e^{-\r^\ast} - \frac{\r^\ast}{10}}\right)\, \r^\ast \eq \log\left(e^{-\r^\ast} + \frac{\r^\ast}{10}\right)
	\]
Observe that $e^{-\r} + \r/10$ is monotone decreasing for $\r \in (0,1)$. Therefore, for $\r\in(0,1)$, the right hand side $\log\left(e^{-\r} + \r/10\right)$ is also monotone decreasing while the left hand side $	\left(1-\frac{1}{1-e^{-\r} - \frac{\r}{10}}\right)\cdot \r$ is monotone increasing. Furthermore, they both attain 0 at $\r=0$.
	\[
	\left.\log\left(e^{-\r} + \frac{\r}{10}\right) \right\vert_{x=0} = 0 = 	\left.\left(1-\frac{1}{1-e^{-\r} - \frac{\r}{10}}\right)\, \r \right\vert_{x=0}
	\]
	Since one is decreasing and other is increasing, there is no $\r^\ast\in (0,1)$ making them equal. 

\subsubsection{Proof of Claim~\ref{subset-bound}}

Since $1+x\le e^x$, we have
\begin{equation}\label{eqn:prod-obs}
\begin{aligned}
		\prod_{i=1}^{k} \left(1- \frac{h}{\revise{\sqrt{n-i}+1}}\right)^i &\le \Exp{ -h \cdot \sum_{i=1}^{k} \frac{i}{(\sqrt{\revise{n}-i}+1)^2}}
\end{aligned}
\end{equation}
The function $x/(\sqrt{\revise{n}-x}+1)^2$ is monotone increasing for $0\le x \le \revise{n}$. We can estimate the inner summation in \eqref{eqn:prod-obs} by evaluating an integral 
\begin{equation} \label{eqn:lem-prod-ub}
\begin{aligned}
	\sum_{i=1}^{k} &\frac{i}{(\sqrt{\revise{n}-i}+1)^2} \ge \int_{0}^{cn}  \frac{x}{(\sqrt{\revise{n}-x}+1)^2} dx\\ 
	&= -cn + 6\sqrt{n} - 6 + \frac{(4c-6)n+6}{1+\sqrt{(1-c)n}} +  2(n-3)\log\left({\frac{1+\sqrt{n}}{1+\sqrt{(1-c)n}}}\right)
\end{aligned}
\end{equation}
When $n\to\infty$, the lower bound above is dominated by $-cn$ and $2n\log\left({\frac{1+\sqrt{n}}{1+\sqrt{(1-c)n}}}\right)$. For the second term, its asymptotic value can be obtained from the Taylor expansion of $\log{x}$, i.e. 
	\[
	\log\left({\frac{1+\sqrt{n}}{1+\sqrt{(1-c)n}}}\right) = \log\left(\frac{1}{\sqrt{1-c}} + O\left(\frac{1}{\sqrt{n}}\right)\right) = -\frac{1}{2}\log(1-c) + O\left(\frac{1}{\sqrt{n}}\right) .
	\]
	By using big-O notation, \eqref{eqn:lem-prod-ub} can be written as
	\[
	\sum_{i=1}^{k} \frac{i}{(\sqrt{\revise{n}-i}+1)^2} \ge -(\log(1-c)+c)\, n + O(n) .
	\]
	Therefore, we can apply this bound to \eqref{eqn:prod-obs} to get
	\[
	\prod_{i=1}^{k} \left(1- \frac{h}{\revise{\sqrt{n-i}+1}}\right)^i \le \Exp{ h\,  (\log(1-c)+c)\, n +O(n)} ,
	\]
which finishes our proof for this claim.


\section{\changes{Proofs for the Recursion Formula}}	\label{sec:missing}

Now we give proofs of technical results from \Cref{sec:recur} about the recursion formula~\cref{def:recur}.

\begin{proof}[\textbf{Proof of \Cref{obsfa}}]
Rewrite $f_{a} = x(1-ax)= - \left(\sqrt{a}x - \frac{1}{2\sqrt{a}}\right)^2 + \frac{1}{4a}$ and then the first two claims follow immediately. The third and fifth claims are, respectively, equivalent to $-ax^{2}=0$ and $-ax^{2}<0$. The fourth claim is because $a\in(0,1]$ implies that $1-ax \in (0,1)$ for $x\in(0,1)$. The sixth claim is no more than a simple computation $af_a(x) = ax(1-ax) = f_1(ax)$. To prove the seventh claim, we evaluate $f_a(x) - f_b(x) = (b-a)x^2$ which is negative if $a >b$ and $x>0$. 
\end{proof}

\begin{proof}[\textbf{Proof of \Cref{xi-bounds}}]
We immediately have from the fourth, fifth and second claims in \Cref{obsfa} that~\eqref{def:recur} generates a decreasing sequence over $(0,1/4a]$. The sixth claim in \Cref{obsfa} implies that the sequence $\{y_{n}\}$ formed by the recurrence $y_{n} = f_{1}(y_{n-1})$ is related to $\{x_{n}\}$ by $y_{n} = ax_{n}$. Therefore, to obtain the other bounds on $x_{n}$ it suffices to bound with $a=1$ and then scale the results by dividing with $a$. Henceforth, assume $a=1$.

We prove by induction that $x_{n} < \frac{1}{n - 1+\frac{1}{x_1}}$. 
The base case of $n=2$ holds because $x_{2} = f_{1}(x_{1}) = x_{1}(1-x_{1}) < x_{1}/(1+x_{1}) = 1/(1 + 1/x_{1})$, where the inequality is from $1-x_{1}^{2} < 1$. Suppose $x_n < 1/(n+\phi)$ for some $n\ge 2$, where for convenience of notation we let $\phi = \frac{1}{x_1}-1$. 
Note that $\phi > 0$ because $x_{1}\in(0,1)$, and hence $1/(n+\phi) \in (0, 1/2)$ and $1/(n+\phi) < 1/n$ for $n\ge 2$. Monotonicity of $f_{1}$ over $(0,1/2)$ from the first claim in \Cref{obsfa} and our induction hypothesis imply that $f_{1}(1/(n+\phi)) > f_{1}(x_{n}) = x_{n+1} $. Since $f_{1}(1/(n+\phi)) = (n+\phi - 1)/(n+\phi)^{2} < 1/(n+1+\phi)$ and $x_{n+1} = f_{1}(x_{n})$, we have arrived at $x_{n+1} < 1/(n+1+\phi)$, which completes the induction for the upper bound.

For the lower bound, define two functions $\phi(\lambda) := \max\{\frac{\lambda}{\lambda-1}, \, \frac{1}{f_{1}(x_{1})} \}$ and $g_{n}(\lambda) := \phi(\lambda)+n\lambda$, with domain $\lambda>1$. We argue, by induction, that $1/g_{n}(\lambda)$ is a parametric lower bound on $x_{n}$. Denote $\phi := \phi(\lambda)$. The base case is $n = 2$. From recurrence formula $x_{2} = f_1(x_1)$ and $\phi \ge 1/f_1(x_1)$, we get $x_{2} \ge 1/\phi$ and thus $x_{2} > 1/(\phi + 2\lambda)$ since $\lambda > 0$. For the inductive step, assume $x_n \ge 1/(\phi + \lambda n)$ is true. \Cref{xi-bounds} says that $x_n\le 1/4$ for all $n\ge 2$ and the first claim in \Cref{obsfa} asserts $f_1(x)$ is monotone increasing on $(0,1/4]$. 
	\(
	x_{n+1} \eq f_1(x_n) \ge f_1\left(\frac{1}{\phi+\lambda n} \right) \eq \frac{(\phi+\lambda n)-1}{(\phi+\lambda n)^2}
	\) 
	 Since $(\phi+n\lambda-1)(\phi+(nk+1)\lambda) = (\phi+n\lambda)^{2} +\lambda(\phi+n\lambda) - (\phi+n\lambda+\lambda)$ and $\lambda(\phi+n\lambda) - (\phi+k\lambda+\lambda) = (\phi+n\lambda)(\lambda-1) - \lambda > 0$ where the inequality is because $\phi+n\lambda > \phi = \lambda/(\lambda-1)$, it follows that $f_{1}(\frac{1}{\phi+\lambda n}) > 1/(\phi+(n+1)\lambda)$. This leads to $x_{n+1} > 1/(\phi+(n+1)\lambda) = \frac{1}{\phi+\lambda (n+1)}$.

Now we show that our claimed lower bound on $x_{n}$ is a lower bound on the supremum of $1/g_{n}(\lambda)$ over $(1,\infty)$, or equivalently an upper bound on $\inf_{\lambda > 1}\, g_{n}(\lambda)$. Denote $\lambda^{\ast} = 1/(1 - f_{1}(x_{1}))$. This value is such that $\lambda^{\ast}/(\lambda^{\ast}-1) = 1/f_{1}(x_{1})$. Since $\lambda/\lambda-1$ is a decreasing function of $\lambda$, we have that the function $\phi$ has non-differentiability at $\lambda^{\ast}$ and $g_{n}$ can be written as
\[
g_{n}(\lambda) \eq \begin{cases}
\frac{\lambda}{\lambda-1} + n\lambda, & 1 < \lambda \le \lambda^{\ast}\smallskip\\
\frac{1}{f_{1}(x_{1})} + n\lambda, & \lambda \ge \lambda^{\ast}.
\end{cases}
\]
Therefore, 
\begin{equation}	\label{twoinf}
\inf_{\lambda > 1}\, g_{n}(\lambda) \eq \min\left\{\inf_{1 < \lambda \le \lambda^{\ast}}\, \frac{\lambda}{\lambda-1} + n\lambda, \; \inf_{\lambda \ge \lambda^{\ast}} \, \frac{1}{f_{1}(x_{1})} + n\lambda \right\},
\end{equation}
and let us evaluate the two infimums separately. The second one is obviously equal to $\frac{1}{f_{1}(x_{1})} + n\frac{1}{1 - f_{1}(x_{1})}$. For the first infimum, it is easy to verify that the function to be minimised is convex in $\lambda$, and the first derivative is $n - 1/(\lambda-1)^{2}$, so that the stationary point is at $\tilde{\lambda} = 1 + 1/\sqrt{n}$. Hence, the minimum value of this function over the real line is its value at $\tilde{\lambda}$, which is equal to $(\sqrt{n}+1)^{2}$. We are interested in the minimum over the interval $(1,\lambda^{\ast}]$. The function has a vertical asymptote $\lambda = 1$ and its value at the other endpoint is $\frac{1}{f_{1}(x_{1})} + \frac{n}{1 - f_{1}(x_{1})}$. It is easy to verify that $(\sqrt{n}+1)^{2} \;\le\;  \frac{1}{f_{1}(x_{1})} + \frac{n}{1 - f_{1}(x_{1})}$ if and only if $((\sqrt{n}+1)f_{1}(x_{1}) - 1)^{2}\ge 0$, which is obviously true. Therefore, the minimum over $(1,\lambda^{\ast}]$ is $(\sqrt{n}+1)^{2}$ if $\tilde{\lambda} \le \lambda^{\ast}$, otherwise the minimum is $\frac{1}{f_{1}(x_{1})} + \frac{n}{1 - f_{1}(x_{1})}$. 
The second value can be upper bounded as follows,
\[
\frac{1}{f_{1}(x_{1})} + n\frac{1}{1 - f_{1}(x_{1})} \;\le\; \frac{1}{f_{1}(x_{1})} + n\frac{\frac{1}{4}}{f_{1}(x_{1})(1 - \frac{1}{4})}
\eq \left(1 + \frac{n}{3} \right)\frac{1}{f_{1}(x_{1})}
\; < \; \frac{n}{f_{1}(x_{1})},
\]
where the first inequality is by applying the following fact that is simple to verify by cross-multiplying denominators,
\[
0 < \xi_{1} \le \xi_{2} < 1 \implies \frac{1}{1-\xi_{1}} \le \frac{\xi_{2}}{\xi_{1}(1-\xi_{2})},
\]
to $\xi_{1} = f_{1}(x_{1})$ and $\xi_{2} = 1/4$ (this upper bound on $x_{2}=f_{1}(x_{1})$ is from \Cref{xi-bounds}), 
and the second inequality is because $n > 1 + n/3$ for $n\ge 2$. Thus, the infimum in \eqref{twoinf} is upper bounded by \revise{$(\sqrt{n}+1)^2$} if $\tilde{\lambda} \le \lambda^{\ast}$, otherwise the bound is $n/f_{1}(x_{1})$.

It remains to simplify the condition  $\tilde{\lambda} \le \lambda^{\ast}$, which becomes $1 + 1/\sqrt{n} \le 1/(1-f_{1}(x_{1}))$. This is equivalent to $f_{1}(x_{1}) \ge 1/(1 + \sqrt{n})$, which rearranges to $n \ge (\frac{1}{f_{1}(x_{1})} - 1)^{2}$.
\end{proof}

The next result we have to establish about the recursion formula is the bounds on the partial sum. To prove the lower bound, let us recall that the $n^{th}$ harmonic number is
\[
H_{n} \define 1 + \frac{1}{2} + \frac{1}{3} + \cdots + \frac{1}{n},
\]
which is the partial sum of the harmonic series $\sum_{i}\frac{1}{i}$. It is obvious from the Riemann approximation of an integral that  
\[
H_{n} \le \log{(n+1)} + 1.
\]
A well-known fact about the harmonic number is that it exhibits a logarithmic growth rate, which can be derived using the Euler-Maclaurin expansion formula.

\begin{lemma}[\cite{boas1971partial}]	\label{harmonic}
$H_{n} = \gamma + \log{n} + \frac{1}{2n} - R_{n}$, where $0\le R_{n} \le \frac{1}{8n^{2}}$ and $\gamma \approx 0.57721$ is the Euler constant.
\end{lemma}

We are now ready to argue our bounds on the partial sums of \cref{def:recur}.

\begin{proof}[\textbf{Proof of \Cref{boundsonsum}}]
For the upper bound, we have
\[
a \sum_{i=1}^{n}x_{i} \,\le\, ax_{1} + \sum_{i=2}^{n}\frac{1}{i-1+\frac{1}{x_{1}}}
\,\le\, ax_{1} + \int_{\frac{1}{x_{1}}}^{\frac{1}{x_{1}}+n-1}\frac{1}{t} \dif t \eq ax_{1} + \log{\left(1 + (n-1)x_{1} \right)},
\]
where the first inequality is using the upper bound in \Cref{xi-bounds} and the second inequality is due to the summation being the right Riemann sum of the decreasing function $t \mapsto 1/(t + \frac{1}{x_{1}})$ over the interval $[0,n-1]$.

For the lower bound, let us denote $n^{\ast} := (\frac{1}{f_{1}} - 1)^{2}$ for convenience. Using the lower bound from \Cref{xi-bounds} for each term in the sequence, the partial sum can be lower bounded as
\[
\begin{split}
a \sum_{i=1}^{n}x_{i} \,\ge\, ax_{1} + \sum_{i=2}^{n}\ell_{i}(x_{1}) 
&\eq ax_{1} \;+\; \sum_{i=2}^{n^{\ast}-1}\frac{f_1(x_{1})}{i} \;+\; \sum_{i=n^{\ast}}^{n} \frac{1}{(\sqrt{i}+1)^2} \\
&\eq ax_{1} \;+\; f_{1}(x_{1})(H_{n^{\ast}-1}-1) \;+\; \sum_{i=n^{\ast}}^{n} \frac{1}{(\sqrt{i}+1)^2} \\
&\,\ge\, ax_{1} \;+\; f_{1}(x_{1})(H_{n^{\ast}-1}-1) \;+\; \int_{n^{\ast}}^{n}\frac{1}{(\sqrt{t}+1)^2}\dif t \\
&\eq ax_{1} \;+\; f_{1}(x_{1})(H_{n^{\ast}-1}-1) \;+\; 2\left[\frac{1}{1+\sqrt{t}} + \log{(1+\sqrt{t})} \right]_{n^{\ast}}^{n}\\
&\,\ge\, ax_{1} \;+\; f_{1}(x_{1})(H_{n^{\ast}-1}-1) \;-\; \frac{2}{1+\sqrt{n^{\ast}}} \;-\; 2\log{\left(1+\sqrt{n^{\ast}}\right)} \\
& \qquad \qquad \qquad \;+\; \frac{2}{1+\sqrt{n}} \;+\; \log{n},
\end{split}
\]
where the first equality is by definition of $\ell_{i}$ in \eqref{piecewise}, the second equality is the definition of harmonic number $H_{n}$, the second inequality is due to the summation being the left Riemann sum of the decreasing function $t \mapsto 1/(1+\sqrt{t})^{2}$ over the interval $[n^{\ast},n]$, and the last inequality is from the fact that $2\log{(1+\sqrt{t})} = \log{(1 + t + 2\sqrt{t})} > \log{t}$. Now we simplify the terms in the middle involving $n^{\ast}$ and argue that they are lower bounded by $\eta(x_{1})$. Since $n^{\ast} := (\frac{1}{f_{1}} - 1)^{2}$, we have $1+\sqrt{n^{\ast}} = 1/f_{1}$ and $\log{(1+\sqrt{n^{\ast}})} = - \log{f_{1}}$. 
Therefore, the middle terms depending on $n^{\ast}$ are
\[
f_{1}(H_{n^{\ast}-1}-1) - 2f_{1} + 2\log{f_{1}} \eq f_{1}(H_{n^{\ast}-1} - 3) + 2\log{f_{1}},
\]
and we have to argue that this is lower bounded by $\eta(x_{1})$. By \Cref{harmonic}, the $n^{th}$ harmonic number can be lower-bounded as 
\[
H_{n} \ge \gamma + \log{n} + \frac{1}{2n} - \frac{1}{8n^{2}} \implies H_{n}-3 \ge \gamma + \log{n} + \frac{1}{2n} - \frac{1}{8n^{2}} - 3 \ge \log{n} - \frac{5}{2},
\]
where the last inequality is because $\gamma > 1/2$ and $1/2n - 1/8n^{2} > 0$. Substituting this lower bound into the middle terms depending on $n^{\ast}$ gives us the lower bound
\[
f_{1}\left(\log{(n^{\ast}-1)} - \frac{5}{2}\right) + 2\log{f_{1}}.
\]
Since $n^{\ast}-1 = (1/f_{1} - 2)/f_{1}$, we have $\log{(n^{\ast}-1)} = \log{(1/f_{1} - 2)} -\log{f_{1}}$, and so the lower bound on the middle terms becomes
\begin{align*}
(2-f_{1})\log{f_{1}} + f_{1}\left(\log{\left(\frac{1}{f_{1} }- 2\right)} - \frac{5}{2}\right) &\eq (2-f_{1})\log{f_{1}} + f_{1}\left(\log{(1- 2f_{1})} - \log{f_{1}} - \frac{5}{2}\right) \\
&\ge 2(1-f_{1})\log{f_{1}} - f_{1}\left(\log{2} + \frac{5}{2} \right) \,=:\, \eta(x_{1}),
\end{align*}
where the inequality is because $\log{(1-2f_{1})} \ge -\log{2}$ due to $f_{1}\le 1/4$ from \Cref{obsfa}.
\end{proof}

\section{\changes{Convergence Proofs for the Bernoulli Sequence}} \label{sec:bernoulli}

\revise{This section establishes the technical results about the Markov chain of Bernoulli r.v. from \Cref{sec:markovchain}. The main result \Cref{bernoulli-main} is proved in multiple parts in \Crefrange{sec:convergence1}{sec:distribution}. Finally, the generalised result \Cref{bernoulli-main3} is proved in \Cref{sec:concentration2} to conclude the paper.}

\subsection{Convergence of Expectation} \label{sec:convergence1}

We prove here the first claim of \Cref{bernoulli-main} that $\E{\YY_{n}}$ converges to $1/\p$. This convergence is a step towards establishing convergence in probability in the next section. 

Since $\E{\YY_{n}} = \E{S_{n}}/\log{n}$, establishing convergence of expectation to $1/\p$ is equivalent to showing that $\p\,\E{S_{n}}/\log{n}$ converges to 1. 
The definition of $S_{n}$ implies that $\E{S_{n}} = \sum_{i=1}^{n}\E{\X_{i}} = \sum_{i=1}^{n}\pp_{i}$. Since the sequence $\{\pp_{i}\}$ follows the recursion $\pp_{i} = f_{\p}(\pp_{i-1})$, we can use the bounding analysis of the recurrence formula that was done earlier in this paper. In particular, applying \Cref{boundsonsum} with $\pp_{i} = x_{i}$ and $a=\p$ gives us the bounds
\begin{equation}	\label{boundsonE}
\p\pp_{1} + \eta(\pp_{1}) + \frac{2}{1+\sqrt{n}} + \log{n} \,\le\,  \p\,\E{S_{n}} \,\le\, \p\pp_{1} + \log{\left(1 + (n-1)\pp_{1} \right)} .
\end{equation}
Divide throughout by $\log{n}$ to get
\[
1 + \frac{\p\pp_{1} + \eta(\pp_{1}) + \frac{2}{1+\sqrt{n}}}{\log{n}} \,\le\,  \p\,\E{\YY_{n}} \,\le\, \frac{\p\pp_{1} + \log{\left(1 + (n-1)\pp_{1} \right)}}{\log{n}} .
\]

We argue that the above lower and upper bounds have limit (as $n\to\infty$) equal to 1, and then the squeeze theorem implies that $\E{\YY_{n}}$ converges to $1/\p$. It is easy to see the limit of the lower bound because $2/(1+\sqrt{n}) \le 2$, and $\p$, $\pp_{1}$, and $\eta(\pp_{1})$ are all constants. Now consider the upper bound. \revise{It is sufficient to show the limit of upper bound is at most 1. Indeed, }
\[
\frac{\p\pp_{1} + \log{\left(1 + (n-1)\pp_{1} \right)}}{\log{n}} \revise{\le \frac{\p\pp_{1} + \log{\left(1 + (n-1) \right)}}{\log{n}} \eq 1 + \frac{\p\pp_{1}}{\log{n}}
}
\]
\revise{Since $\p, \pp_{1}$ are constants, the limit of $1 + \frac{\p\pp_{1}}{\log{n}}$ is 1. }



\subsection{Convergence in Probability}	\label{sec:convergence2}

This section proves the second claim of \Cref{bernoulli-main} that $\YY_{n}$ concentrates to $1/\p$. We first argue an upper bound on the second moment of $S_{n}$, and then use this bound in conjunction with two key technical results on sequences of random variables to finally deduce our claim on the concentration of $\YY_{n}$.

\subsubsection{Second Moment of Partial Sum}

The definition $S_n = Y_1+Y_2+\cdots+Y_n$ implies that 
\[
S_{n}^{2} \eq \sum_{i=1}^{n}\X_{i}^{2} + 2\sum_{i < j}\X_{i}\X_{j} \eq \sum_{i=1}^{n}\X_{i} + 2\sum_{i < j}\X_{i}\X_{j},
\]
where the second equality is due to each $\X_{i}$ being a Bernoulli r.v. Therefore, linearity of expectation and $\E{S_{n}} = \sum_{i=1}^n \E{\X_i}$ means that the second moment of $S_{n}$ is
\begin{equation}\label{eqn:square-expected}
\E{S_{n}^2} \eq \E{S_{n}}  + 2 \sum_{i < j}\E{(\X_i \X_j)}.
\end{equation}
The next lemma upper bounds the expected value of each product term $\X_{i}\X_{j}$.

\begin{lemma}\label{square_expected_upper_bound}
For $n\ge 2$,
	\[
	\E{S_{n}^2} \le \frac{2\r}{\rr }\E{S_{n}}+ \left(\frac{\log{n}}{\rr }\right)^2 .
	\]
\end{lemma}
\begin{proof}
If $i\le j-2$, the joint probability $\P{\X_i\X_j=1}$ is at most $\pp_{i}\pp_{j-1}$. Indeed, by \Cref{ineqOfCondProb},
	\[
	\P{\X_i \X_j = 1} \eq \P{\X_{i}=1, \, \X_{j}=1} \P{\X_j=1}[\X_i=1]\P{\X_i=1} \le \pp_{i}\pp_{j-1}
	\]	
Furthermore, for two dependent variables $\X_{j-1}$ and $\X_{j}$,
	\[
	\P{\X_{j-1}\X_{j}=0} = \P{\X_j=0} + 
	\P{\X_{j-1}=0}\P{\X_j=1}[\X_{j-1}=0]=1-\pp_j+\pp_{j-1}(1-\pp_{j-1}) 
	\]
The recurrence relation implies $\pp_j=\pp_{j-1}(\rr r \pp_{j-1})$. Thus,
	\[
	\P{\X_i \X_j=1} \le 
	\left\{
	\begin{aligned}
		&\pp_i\pp_{j-1},& &i\le j-2\\
		&\left(\frac{1}{\rr }-1\right)(\pp_{j-1}-\pp_j),&&i=j-1
	\end{aligned}
	\right.
	\]
	We are ready to compute an upper bound for $\sum_{i < j}\E{\X_i \X_j}$.
	\[
	\begin{aligned}
		\sum_{i < j}\E{\X_i \X_j} &= \sum_{j=2}^n\sum_{i=1}^{j-1} \E{\X_i \X_j} =\sum_{j=2}^n \E{\X_{j-1}\X_j} + \sum_{i=3}^n\sum_{i=1}^{j-2}\E{\X_i \X_j}\\
		&\le\left(\frac{1}{\rr }-1\right) \sum_{i=2}^n (\pp_{j-1}-\pp_j) + \sum_{j=3}^n\sum_{i=1}^{j-2}\pp_i\pp_{j-1}\\
		&\le\left(\frac{1}{\rr }-1\right)(\pp_1-\pp_n) + \sum_{j=3}^n\sum_{i=1}^{j-2}\pp_i\pp_{j-1}.
	\end{aligned}
	\]
	Note that $\E{S_{n}} = \sum_{j=1}^n \E{X_j} = \sum_{j=1}^n \pp_j$. Then (\ref{eqn:square-expected}) can be written as
	\[
	\E{S_{n}^2} =\E{S_{n}}+ \frac{2\r}{\rr }(\pp_1-\pp_n)  + 2\sum_{j=3}^n\sum_{i=1}^{j-2}\pp_i\pp_{j-1} .
	\]
	From \Cref{edgeprob}, $\rr \pp_i$ has an upper bound $1/i$, and so
	\[
	\sum_{j=3}^n\sum_{i=1}^{j-2}\pp_i\pp_{j-1}\le \frac{1}{\rr }\sum_{j=3}^n  \pp_{j-1} \sum_{i=1}^{j-2} \frac{1}{i} .
	\]
	The partial sum of harmonic series $\sum_{i=1}^{j-2}\frac{1}{i} \le \log(j-1) +1$. Therefore,
	\[
	\sum_{j=3}^n\sum_{i=1}^{j-2}\pp_i\pp_{j-1} \le \frac{1}{\rr }\sum_{j=3}^n\pp_{j-1} + \frac{1}{\rr }\sum_{j=3}^{n}\log(j-1)\pp_{j-1} = \frac{\E{S_{n}}-\pp_1-\pp_2}{\rr }+ \frac{1}{\rr }\sum_{j=3}^{n}\log(j-1)\pp_{j-1} .
	\]
	Then, we reuse the bound $ \pp_{j} \le 1/(\rr  j)$ from \Cref{edgeprob},
	\[
	\frac{1}{\rr }\sum_{j=2}^{n-1}(\log{j})\pp_{j} \le \frac{1}{\rr ^2}\sum_{j=2}^{n-1} \frac{\log{j}}{j}  \le \frac{1}{\rr ^2} \int_{1}^{n}\frac{\log{t}}{t}dt = \frac{1}{2}\left(\frac{\log{n}}{\rr }\right)^{2}.
	\]	
	Combining the two inequalities above, we obtain
	\[
	\sum_{j=3}^n\sum_{i=1}^{j-2}\pp_i\pp_{j-1} \le \frac{\E{S_{n}}-\pp_1-\pp_2}{\rr } + \frac{1}{2}\left(\frac{\log{n}}{\rr }\right)^2 .	
	\]
	It follows that
	\[
	\begin{aligned}
		\E{S_{n}^2} &\le \frac{2\r}{\rr }\E{S_{n}}-2\pp_1-\frac{2}{\rr }\pp_2- \frac{2\r}{\rr }\pp_n + \left(\frac{\log{n}}{\rr }\right)^2\\
		&\le \frac{2\r}{\rr }\E{S_{n}}+ \left(\frac{\log{n}}{\rr }\right)^2 ,
	\end{aligned}
	\]
	which is the desired result.
\end{proof}

%

\subsubsection{Two Lemmas on Random Sequences}

Let $\{\XX_{n}\}$ be a sequence of random variables all of which are defined on the same probability space. The following result about the terms in the sequence concentrating to their mean  is a straightforward consequence of a classic concentration inequality.

\begin{lemma}	\label{convergence1}
Suppose that $\XX_{n} \ge 0$ a.s. for all $n$, and the mean $\mu_{n}\in (0,\infty)$ and variance $\sigma^{2}_{n}$ of $\XX_{n}$ obey $\sigma^{2}_{n} = o(\mu^{2}_{n})$. Then, $\XX_{n}/\mu_{n}$ concentrates to 1.
\end{lemma}
\begin{proof}
Chebyshev's inequality applied to $\XX_{n}$ tells us that $\P{\abs{\XX_{n} - \mu_{n}} \ge k\sigma_{n}} \le 1/k^{2}$ for any $k > 0$. Taking $k = \epsilon \mu_{n}/\sigma_{n}$ for arbitrary $\epsilon > 0$ gives us $\P{\abs{\XX_{n} - \mu_{n}} \ge \epsilon \mu_{n}} \le \sigma^{2}_{n}/\left(\epsilon^{2}\mu^{2}_{n}\right)$. The assumption of $\mu_{n}$ being  positive makes the probability equal to $\P{\abs{\XX_{n}/\mu_{n} - 1} \ge \epsilon }$. Now taking limit on both sides leads to
\[
\lim_{n\to\infty} \P{\abs{\frac{\XX_{n}}{\mu_{n}} - 1} \ge \epsilon } \:\le\: \lim_{n\to\infty} \frac{\sigma^{2}_{n}}{\epsilon^{2}\mu^{2}_{n}} .
\]
Since $\sigma^{2}_{n} = o(\mu^{2}_{n})$, the right-hand side goes to zero, which implies that the left-hand side is also zero and therefore, $\XX_{n}/\mu_{n}\overset{p}{\longrightarrow} 1$.
\end{proof}

The second useful result is a special case of Slutsky's theorem, or more generally the fact that convergence in probability is preserved under multiplication \cite[cf.][Theorem III.3.6]{ccinlar2011probability}. 

\begin{lemma}	\label{convergence2}
Let $\{\ZZ_{n}\}$ be a converging sequence of positive reals with limit $\tau$ and such that $\XX_{n}/\ZZ_{n} \overset{p}{\longrightarrow} \lambda$ for some constant $\lambda > 0$. Then, $\XX_{n} \overset{p}{\longrightarrow} \lambda\tau$.
\end{lemma}

Note that the sequence $\{\YY_{n}\}$ is of positive reals and hence it trivially converges in probability.

\subsubsection{Deducing the Concentration Result}

\begin{lemma} \label{ratio_of_expected}
$\displaystyle\lim_{n\to \infty}\frac{\E{S_{n}^2}}{(\E{S_{n}})^2} \eq 1$.
\end{lemma}
\begin{proof}
Since the variance of $S_{n}$ is equal to $\E{S_{n}^{2}} - (\E{S_{n}})^{2}$ and this variance is nonnegative, we know that the ratio $\E{S_{n}^2}/(\E{S_{n}})^2$ is at least 1. So, it remains to show that the limit is upper bounded by 1. For this proof, we only need a weaker lower bound for $\E{S_{n}}$ in \eqref{boundsonE} by dropping out $ \frac{2}{1+\sqrt{n}}$ to get
    \[
    \E{S_{n}}\ge \pp_1 + \frac{1}{\p}\left( \eta(p_{1}) + \log{n}\right),
    \]
which implies that
	\[
	(\E{S_{n}})^2 \ge \pp_1^2+\frac{2\pp_1}{\p}\left( \eta(p_{1}) + \log{n}\right)+\frac{1}{\p^2}\left( \eta(p_{1}) + \log{n}\right)^2 .
	\]
Combining this lower bound with the upper bound on $\E{S_{n}^2}$ from \Cref{square_expected_upper_bound} leads to
	\[
	\frac{\E{S_{n}^2}}{(\E{S_{n}})^2} \,\le\, \frac{2\r}{\rr \E{S_{n}}} +\frac{ \left(\frac{\log{n}}{\rr }\right)^2}{\pp_1^2+\frac{2\pp_1}{\rr }\left( \eta(w) + \log{n}\right)+\frac{1}{\rr ^2}\left( \eta(w) + \log{n}\right)^2} .
	\]
Note that $\lim_{n\rightarrow \infty} \E{S_{n}} = \infty$ and thus the first term goes to $0$. The limit of second term is dominated by coefficients of $(\log{n})^2$. Then 
	\[
	\lim_{n\rightarrow \infty}\frac{\E{S_{n}^2}}{(\E{S_{n}})^2} \le \lim_{m\rightarrow \infty} \frac{ \left(\frac{\log{n}}{\rr }\right)^2}{\frac{1}{\rr ^2} (\log{n})^2} =1,
	\]
which completes the proof of this claim.		
\end{proof}

This implies that the variance of $S_{n}$ is equal to $o((\E{S_{n}})^{2})$. Therefore, by \Cref{convergence1}, $S_{n}/\E{S_{n}}$ concentrates to 1. Equivalently, $\YY_{n}/\E{\YY_{n}}$ concentrates to 1.  From the convergence of expectation result in \Cref{sec:convergence1} we know that $\E{\YY_{n}}$ converges to $1/\p$. Then, applying \Cref{convergence2} with $\XX_{n} = \YY_{n}$ and $\ZZ_{n} = \E{S_{n}}/\log{n}$ yields the desired result that $\YY_{n}$ concentrates to $1/\p$.

\subsection{Distribution Function}	\label{sec:distribution}

The distribution of $\YY_{n}/\E{\YY_{n}}$ is the same as that of $S_{n}/\E{S_{n}}$. Since $S_{n}\ge 0$ a.s., Markov's inequality immediately provides the upper bound $\P{\frac{S_{n}}{\E{S_{n}}} \,\ge\, 1-\theta} \le 1/(1-\theta)$, which then also holds in the limit as $n\to\infty$. For the lower bound, we use  the Paley--Zygmund concentration inequality which gives us
	\[
	\P{S_{n} \ge (1-\theta)\,\E{S_{n}}} \ge \theta^2\frac{(\E{S_{n}})^2}{\E{S_{n}^2}}.
	\]
Taking $n\rightarrow \infty$, \Cref{ratio_of_expected} yields $\lim_{n\rightarrow \infty} \P{S_{n} \ge (1-\theta)\E{S_{n}}} \ge \theta^2$.

\subsection{Proof of Corollary~\ref{bernoulli-main3}}	\label{sec:concentration2}

Choose any $ \epsilon > 0$ such that $\epsilon < \min\{1-\p,\, \p/3  \}$. Since both $\{a_n\}$ and $\{b_n\}$ converge to $\p$, there exists an integer $N$ such that for every $n\ge N$ we have
	\[
	\p-\epsilon<a_n<\p+\epsilon ~\text{ and }~ \p-\epsilon<b_n<\p+\epsilon.
	\]
Let $\p^+ := \p+\epsilon$ and $\p^- := \p-\epsilon$. From the seventh property of $f_a$ in \Cref{obsfa},
	\begin{equation}\label{eqn:bd-bernoulli-main1}
	f_{\p^+}(\pp_n)\le f_{a_n}(\pp_n) \le \pp_{n+1}\le f_{b_n}(\pp_{n})\le f_{\p^-}(\pp_n),\qquad \forall n\ge N.
	\end{equation}
	 Construct two new sequences $\{l_i\}_{i\ge N}$ and $\{u_i\}_{i\ge N}$ in which $l_N=u_N=\pp_N$ and subsequent elements for $n\ge N+1$ are generated by the recurrence formulas as follows.
	\[
	l_n=f_{\p^+}(l_{n-1}), \quad u_n = f_{\p^-}(u_{n-1}).
	\]
We want to show $l_n\le \pp_n\le u_n$ for all $n\ge N$. Note that when $n=N$, it's obviously true by our construction. We now prove the inequality for $n\ge N+1$ by induction. The base case $l_{N+1}\le \pp_{N+1}\le u_{N+1}$ follows directly from (\ref{eqn:bd-bernoulli-main1}). Suppose $l_n\le \pp_n\le u_n$ for some $n\ge N+1$. By \eqref{eqn:bd-bernoulli-main1},
	\[
	u_{n} \le f_{\p^-}(u_{n-1})\le \max_{x\in (0,1)} f_{\p^-}(x) = \frac{1}{4\p^-} \le \frac{1}{2\p^+},\qquad \forall n\ge N+1
	\]
	where the last inequality follows from our choice of $\epsilon$. As a result,   $l_n,\pp_n,u_n\in (0,1/4\p^+]$. The first property in \Cref{obsfa} shows $f_{\p^+}$ and $f_{\p^-}$ are both monotone increasing on $(0,1/2\p^+]$. Therefore,
	\[
	l_{n+1}=f_{\p^+}(l_n)\le f_{\p^+}(\pp_n)\le \pp_{n+1}\le f_{\p^-}(\pp_n)\le f_{\p^-}(u_n)=u_{n+1}
	\]
	We extend $\{l_n\}_{n\ge N}$ and $\{u_n\}_{n\ge N}$ by defining $l_i = p_i = u_i$ for $1\le i\le N-1$. We are interested in the partial sum of these two sequences i.e. $L_n = \sum_{i=1}^n l_n$ and $U_n =\sum_{i=1}^nu_i$. From the previous arguments,  
	\[
	\lim_{n\rightarrow \infty}\frac{L_n}{\log{n}}\le \lim_{n\rightarrow \infty}\frac{\E{S_n}}{\log{n}} \le \lim_{n\rightarrow \infty}\frac{U_n}{\log{n}}
	\]
	 Actually, the limits of the left and right hand sides are equal to $1/\p^+$ and $1/\p^-$ respectively. To see these, we first define two sequence $\{l'_n\}_{n\ge 1}$ and $\{u'_n\}_{n\ge 1}$ such that $l'_{i+1} = f_{\p^+}(l'_i)$ and $u'_{i+1}=f_{\p^-}(u'_i)$. Furthermore, $l'_i = l_i$ and $u'_i=u_i$ for all $i\ge N$. In other words, we extend  $\{l_n\}_{n\ge N}$ and $\{u_n\}_{n\ge N}$ by two sequences $\{l'_n\}_{n\ge 1}$ and $\{u'_n\}_{n\ge 1}$  constructed completely from recurrence formulas. Note that $\{l'_n\}_{n\ge 1}$ and $\{u'_n\}_{n\ge 1}$  are well-defined because the recurrence formulas $f_{\p^+}$ and $f_{\p^-}$ are both monotone increasing on $(0,1/2\p^+]$ and $l_N, u_N \in (0,1/4\p^+]$. It follows that there exist inverse elements $l'_{N-1}, u'_{N-1}\in (0,1/4\p^+]$ such that $l_N= f_{\p^+}(l'_{N-1})$ and $u_N=f_{\p^-}(u'_{N-1})$. We continue this procedure to get $l'_{N-2}, u'_{N-2}$ such that $l'_{N-1}= f_{\p^+}(l'_{N-2})$ and $u'_{N-1}=f_{\p^-}(u'_{N-2})$. Eventually, we can construct $\{l'_n\}_{n\ge 1}$ and $\{u'_n\}_{n\ge 1}$. From \Cref{bernoulli-main}, the partial sums $L'_n = l'_1 +l'_2+\ldots+l'_n$ and $U'_n=u'_1+u'_2+\ldots+u'_n$ satisfy
	\[
	\lim_{n\rightarrow \infty}\frac{L'_n}{\log{n}} = \frac{1}{\p^+} \text{ and } \lim_{n\rightarrow \infty}\frac{U'_n}{\log{n}} =\frac{1}{\p^-}
	\]
	There are only the first $N$ elements in $\{l_n\}_{n\ge 1}$ and $\{u_n\}_{n\ge 1}$ are different from $\{l'_n\}_{n\ge 1}$ and $\{u'_n\}_{n\ge 1}$. The sum of the first $N$ elements divided by $\log n$ goes to 0 as $n\to\infty$. Consequently,
	\[
	\lim_{n\rightarrow \infty}\frac{L_n}{\log{n}} = \frac{1}{\p^+} \text{ and } \lim_{n\rightarrow \infty}\frac{U_n}{\log{n}} =\frac{1}{\p^-}
	\]
This implies that
	\[
	\frac{1}{\p+\epsilon}=\frac{1}{\p^+}\le\lim_{n\rightarrow \infty}\frac{\E{S_n}}{\log{n}}\le \frac{1}{\p^-} = \frac{1}{\p-\epsilon}
	\]
and then taking $\epsilon \rightarrow 0$ yields our desired claim $\lim_{n\rightarrow \infty}\frac{\E{S_n}}{\log{n}} =\frac{1}{\p}$.

\ifdefined \informs
	\ACKNOWLEDGMENT{
\else
	\subsection*{Acknowledgments}
\fi
\changes{We thank two referees for their careful reading of the original manuscript and suggestions on clarifying technical details.}
\ifdefined \informs
	}
\else
\fi

%
%
%

{
\newrefcontext[sorting=nyt]
\printbibliography
}

\end{document}